\documentclass[a4paper]{scrartcl}
	\usepackage[utf8]{inputenc}
	\usepackage[T1]{fontenc}
	\usepackage[UKenglish]{babel}
	\usepackage{lmodern}
	\usepackage{scrpage2}				
		\clearscrheadfoot
		\ohead{\thepage}
		\ihead{\rightmark}
		\automark[section]{chapter}
		\setheadsepline{1pt}
	\addtokomafont{sectioning}{\rmfamily}   
	\usepackage{graphicx}
	\usepackage{subfigure}
	\usepackage[font=scriptsize]{caption}
	\usepackage{epstopdf}
	\usepackage{multicol}
	\usepackage{float}
	\usepackage{wrapfig}
	
	\usepackage{amsmath, amssymb, amsfonts, amsthm, mathtools}	
	\usepackage{hyperref}
	\usepackage[capitalize, nameinlink, noabbrev]{cleveref}
	\crefname{equation}{}{}
	\usepackage{url}
	
\newtheoremstyle{defi}
{10pt}
{10pt}
{\itshape}
{}
{\normalfont \bfseries}
{:}
{ }
{}

\theoremstyle{defi}
\newtheorem{defi}{Definition}

\newtheoremstyle{exam}
{10pt}
{10pt}
{}
{}
{\normalfont \bfseries}
{:}
{ }
{}

\theoremstyle{exam}

\newtheoremstyle{rema}
{10pt}
{10pt}
{}
{}
{\normalfont \bfseries}
{:}
{ }
{}

\theoremstyle{rema}
\newtheorem{rema}[defi]{Remark}

\newtheoremstyle{theo}
{10pt}
{10pt}
{\itshape}
{}
{\normalfont \bfseries}
{:}
{ }
{}

\theoremstyle{theo}
\newtheorem{theo}[defi]{Theorem}

\newtheoremstyle{lemm}
{10pt}
{10pt}
{\itshape}
{}
{\normalfont \bfseries}
{:}
{ }
{}

\theoremstyle{lemm}

\newtheoremstyle{coro}
{10pt}
{10pt}
{\itshape}
{}
{\normalfont \bfseries}
{:}
{ }
{}

\theoremstyle{coro}

\newtheoremstyle{algo}
{10pt}
{10pt}
{}
{}
{\normalfont \bfseries}
{:}
{ }
{}

\theoremstyle{algo}

\newcommand{\real}{\mathbb{R} }

\newcommand{\sint}{\mathbb{Z} }
\newcommand{\uint}{\mathbb{N} }
\newcommand{\ball}{\mathbb{B}}

\newcommand{\lon}{\varphi}
\newcommand{\lat}{\theta}
\newcommand{\era}{\varepsilon^r}
\newcommand{\ephi}{\varepsilon^\lon}
\newcommand{\ete}{\varepsilon^t}

\newcommand{\pdervlon}{\frac{\partial}{\partial \lon}}
\newcommand{\pdervt}{\frac{\partial}{\partial t}}
\newcommand{\pdervnu}{\frac{\partial}{\partial \nu}}

\newcommand{\belt}{\Delta^*}

{
\setcounter{MaxMatrixCols}{15}

\addto\captionsUKenglish{

}

\begin{document}
\parindent 0pt

\setlength\abovedisplayshortskip{0pt}
\setlength\belowdisplayshortskip{0pt}
\setlength\abovedisplayskip{4pt}
\setlength\belowdisplayskip{4pt}

\begin{center}
\begin{LARGE}
\textbf{Vectorial Slepian Functions on the Ball}\\
\end{LARGE}
\vspace{0.5cm}
V. Michel, S. Orzlowski, N. Schneider\\
Geomathematics Group, Department of Mathematics, University of Siegen, Germany
\end{center}

\paragraph{Abstract} 

Due to the uncertainty principle, a function cannot be simultaneously limited in space as well as in frequency.  The idea of Slepian functions in general is to find functions that are at least optimally spatio-spectrally localised. Here, we are looking for Slepian functions which are suitable for the representation of real-valued vector fields on a three-dimensional ball.
We work with diverse vectorial bases on the ball which all consist of Jacobi polynomials and vector spherical harmonics. Such basis functions occur in the singular value decomposition of some tomographic inverse problems in geophysics and medical imaging, see \cite{MichelOrzlowski2016}. Our aim is to find bandlimited vector fields that are well-localised in a part of a cone whose apex is situated in the origin. Following the original approach towards Slepian functions, the optimisation problem can be transformed into a finite-dimensional algebraic eigenvalue problem. The entries of the corresponding matrix are treated analytically as far as possible. For the remaining integrals, numerical quadrature formulae have to be applied. The eigenvalue problem decouples into a normal and a tangential problem.
The number of well-localised vector fields can be estimated by a Shannon number which mainly depends on the maximal radial and angular degree of the basis functions as well as the size of the localisation region. 
We show numerical examples of vectorial Slepian functions on the ball, which demonstrate the good localisation of these functions and the accurate estimate of the Shannon number. \\



\section{Introduction}

An example for a tomographic problem in mathematics with an unknown vector field can be obtained from medical imaging, more explicitly neuroscience. 
The effects of the neural currents in the brain can be measured by means of magnetoencephalography, MEG, or electroencephalography, EEG. As it was shown in \cite{MichelOrzlowski2016}, the inversion of EEG-MEG-data is mathematically related to other inverse problems, for example, in the geosciences.
In a mathematical model, the unknown currents can be represented as vectorial functions on a ball. 
Note that the neural currents exist in the cerebrum which is a proper subset of the interior of the scalp. 

For tomographic inverse problems, where the unknown is a function on the ball, several different approaches have been used up to now for the construction of localised trial functions including wavelet- and spline-based methods, see, for instance, \cite{Tomo1, Tomo4, FokasHaukMichel2012, Tomo6} \cite{Tomo7, Tomo8, Michel2015_2_Handbook, Tomo10, Tomo11, Tomo12}.
The more challenging task of finding localised vectorial functions on the ball has, however, only rarely been tackled so far, although vectorial tomographic problems on balls and similar geometries occur in a series of applications (see e.g. \cite{Fokas2, Fokas3, TomoApp1, FokasHaukMichel2012, Fokas1, TomoApp2, TomoApp3, TomoApp5, Sarvas, TomoApp4}).

One out of many types of localised trial functions is represented by Slepian functions.
Such functions have been constructed in several settings. At first, localised functions on the real axis have been presented by Landau, Pollak and Slepian in \cite{LandauPollakSlepian2, LandauPollakSlepian4, LandauPollakSlepian1}. Later, scalar Slepian functions on the sphere, for example \cite{Sansoeetal, Simons2006}, and on the ball, for example \cite{Khalid2016}, as well as Slepian vector fields on the sphere, for example \cite{JahnBokor2014, Simons2014}, were studied. Slepian functions have also been used for inverse problems on the sphere, see, for example \cite{SlepiansNeu1, SlepiansNeu2}.
Other studies on Slepian functions, for example \cite{Gruenbaumetal, ManiarMitra, Miranian}, emphasize more specific details. With respect to the MEG-problem, vectorial Slepian functions have been discussed in \cite{ManiarMitra2}. 
In this approach, a (physically motivated) reproducing kernel and the sensor positions are used to generate vectorial basis functions on the ball (in the nomenclature of \cite{Tomo1, BerkelFischerMichel, FokasHaukMichel2012}, these functions could be interpreted as spline basis functions). The Slepian functions are then derived as optimally localised expansions in this `spline' basis. We will elaborate here an alternative ansatz by using some known orthonormal bases on the ball. The reason is that, in a forthcoming publication, we want to use Slepian functions in some regularisation algorithms (see \cite{Fischer, DoreenDiss, Michel2015Handbook, MichelOrzlowski2017, MichelTelschow2016, RogerDiss}), which profit numerically from analytic expressions for the application of the forward operator to the used trial functions. We expect to be able to derive such expressions for basis systems of the considered types. To the knowledge of the authors, such vectorial Slepian functions on the ball are new.

The uncertainty principle forbids that a function can be limited in space and frequency at the same time. Thus, for Slepian functions in general, the boundedness in frequency is usually fixed. Then optimally localised functions can be considered. The general approach to the previously studied Slepian functions can be summed up as follows for the case of spatial localisation. In principle, we consider functions on a domain $D$ which shall be localised to a subdomain $R$. Then the quotient of the $\mathrm{L}^2 (R)$-norm and the   $\mathrm{L}^2 (D)$-norm is considered and is called the energy ratio. It is assumed that the Fourier expansion of the Slepian functions with respect to a chosen orthonormal basis has a finite number of terms (bandlimited expansion) -- for this reason, the use of a different (finite) basis essentially changes the obtained Slepian functions. By inserting this expansion into the energy ratio, the problem is equivalently formulated as a finite-dimensional algebraic eigenvalue problem of a so-called localisation matrix. The expansion coefficients of one Slepian function form the entries of one eigenvector of this matrix. The corresponding eigenvalue equals the energy ratio. An entry of this matrix is defined as the respective inner product of two basis functions. Therefore, the Gramian matrix of the eigenvalue problem can be calculated independently of the Slepian functions. Its solution yields a set of functions, which have a finite expansion in the chosen basis and a related localisation measure. Hence, their localisation in $R$ can be compared pairwise. Functions are better localised if their related eigenvalue has a higher value. Because of the principal axis theorem, the eigenvectors constitute a basis for the respective Euclidean vector space. Hence, the Slepian functions form an alternative basis in the respective function space due to Parseval's identity. This approach was used in most publications regarding Slepian functions. It is also convenient for our case of real-valued vector fields on the ball.  

This paper is structured as follows.
\cref{Prel} points out some common geomathematical notation. After that, three different vector bases on the ball are defined in \cref{Basis}. They all consist of Jacobi polynomials and vector spherical harmonics. In \cref{Sleps}, the bases are used to formulate the bandlimited orthogonal expansion of a Slepian vector field. In \cref{Sleps_1}, following the original approach towards Slepian functions, the optimisation problem is rearranged to a finite-dimensional algebraic eigenvalue problem. The entries of the Gramian matrix are treated analytically as far as possible. The analysis can be found in \cref{EntriesApp}. For practical purposes, some specifications mainly regarding the localisation region are made in \cref{Sleps_2}. The aim is to find vector fields that are well-localised in a part of a cone whose apex is situated in the origin. In \cref{Sleps_3}, the number of well-localised vector fields is predicted by a Shannon number. Finally, some Slepian functions are computed numerically and illustrated in \cref{Results}. It can be seen that the obtained functions are, indeed, strongly localised in the chosen region.

\section{Preliminaries} \label{Prel}
In the sequel, we sum up the definitions needed for the formulation of the Slepian localisation problem. In this paper, the common geomathematical notation will be used. It is introduced, for example, in \cite{FreedenGutting2013, FreedenSchreiner2009, Michel2013}. The terms 
\begin{align*}
\Omega \coloneqq \left\{ \xi \in \real^3 \colon \vert \xi \vert = 1 \right \}, \qquad 
\ball \coloneqq \left\{ x \in \real^3 \colon \vert x \vert \leq 1 \right\}
\end{align*} stand for the unit sphere and the unit ball, respectively. For a measurable subset $S \subseteq \real^3$, we define $\mathrm{l}^2(S) \coloneqq \mathrm{L}^2(S, \real^3)$. S can also be a surface in $\real^3$. A parameterisation of any point $x$ in $\real^3$ is given by
\begin{align}
x(r, \lon, t) = \Big( r \sqrt{1-t^2} \cos(\lon),\ r \sqrt{1-t^2} \sin(\lon),\ rt \Big)^\mathrm{T} \label{PolarParam}
\end{align} for $r \in \ [0, \infty[,\ \lon \in [0,2\pi[$ and $t \in [-1,1], \ t = \cos(\lat),\ \lat \in [0,\pi]$. Note that for any point $\xi \in \Omega$, the radial coordinate equals 1.

The surface gradient operator $\nabla^*$ represents the angular part (up to a factor $\tfrac{1}{r}$ for the length $r$ of a point) of the gradient operator $\nabla$. Furthermore, the surface gradient operator always yields a tangential field. The Beltrami operator $\belt$ is correspondingly the angular part of the Laplace operator $\Delta$ (up to a factor $\tfrac{1}{r^2}$ for the length $r$ of a point). Moreover, the surface curl $\mathrm{L}^*$ is defined via $(\mathrm{L}^* F) (\xi) \coloneqq \xi \times \nabla^* F(\xi)$. For further details, see, e.g., \cite[pp.~37-38]{FreedenSchreiner2009} or \cite[p.~87]{Michel2013} and \cref{EntriesApp} of this paper. Note that there exist versions of Green's theorems with these surface operators, see, for instance, \cite[pp.~40-41]{FreedenSchreiner2009}.

\section{Orthonormal Basis Systems on the Ball} \label{Basis}
Note that we consider here a modelling based on polar coordinates, because structures in a human brain (and inside the Earth) are, roughly speaking, layers with almost spherical boundaries. For this reason, we need corresponding basis systems, as they can be found in \cite{Tomo2, Tomo3, MichelOrzlowski2016, Tscherning1996} (see also the references therein). For orthogonal polynomials on $\ball$ with cartesian coordinates, see \cite{DunkelXu}.

To obtain a vectorial orthonormal basis on the ball $\ball$, we transform the scalar functions by means of the operators defined, for instance, in \cite[p.~218]{FreedenGutting2013}. Hence, we use combinations of diverse Jacobi polynomials and vector spherical harmonics. The latter are defined as follows for $\xi \in \Omega$:
\begin{align*}
y_{n,j}^{(1)} (\xi) \coloneqq \xi Y_{n,j}(\xi),\ 
y_{n,j}^{(2)} (\xi) \coloneqq \sqrt{\frac{1}{n(n+1)}} \nabla^*_\xi Y_{n,j}(\xi),\ 
y_{n,j}^{(3)} (\xi) \coloneqq \sqrt{\frac{1}{n(n+1)}} \mathrm{L}^*_\xi Y_{n,j}(\xi),
\end{align*}
where $Y_{n,j}(\xi)$ are scalar spherical harmonics, for instance fully normalised spherical harmonics as orthonormal polynomials on $\Omega$.
These are given as follows, see, for instance, \cite[p.~142]{FreedenGutting2013}: for $n \in \uint_0,\ j \in \sint,\ \vert j \vert \leq n$ as well as polar coordinates $\lon \in [0, 2\pi[$ and $ t \in [-1, 1]$ 
we define 
\begin{align}
\label{FNSH}
Y_{n,j}(\xi( \lon,t) ) &\coloneqq \sqrt{ \frac{(2n+1)}{2}\ \frac{(n-|j|)!}{(n+|j|)!} }\ P_{n,|j|}(t)\ \frac{1}{\sqrt{2\pi}}\ \left\{ \begin{matrix}
\sqrt{2}\cos \left(j\lon\right), &j < 0, \\ 
1, &j=0,\\ 
\sqrt{2}\sin \left(j\lon\right), &j > 0
\end{matrix} \right. \\
&\eqqcolon b_{n,j} P_{n,|j|}(t)\ \frac{1}{\sqrt{2\pi}}\ c_j(\varphi). \notag
\end{align}  
The functions $P_{n,|j|}$, $n \in \uint_0,\ j=-n,...,n,$ stand for associated Legendre functions.  
Note that $y_{n,j}^{(1)}$ is defined for all non-negative integers $n$, but $y_{n,j}^{(2)}$ and $y_{n,j}^{(3)}$ are only defined for positive integers $n$. In the following, this will be denoted by $$n \geq 0_i \textrm{ for } 0_i \coloneqq 1-\delta_{i1}$$ for the Kronecker Delta $\delta_{i1}$ and $i \in \{ 1,2,3\}.$ 
Then a basis of $\mathrm{l}^2(\ball)$ is given as follows.

\begin{defi}\label{gmnj}
Let $\beta > 0$ be the radius of a given ball. Further, fixed integers are given by $i=1,2,3,\ m \in \uint_0,\ n \in \uint_{0_i}$ and $j=-n,\dots,n$. The functions $P_m^{(\alpha, \beta)}$ stand for the Jacobi polynomials. The functions $y_{n,j}^{(i)}$ denote vector spherical harmonics. 
At last, let any point $x \in \ball$ be given by $ x  = r\xi$ with $r = |x| $ and $\xi \in \Omega$. 
The system $\mathrm{I}$ is defined for $x \in \ball$ by
\begin{align*}
g_{m,n,j}^{(\mathrm{I},i)}\ (r\xi) 
&\coloneqq F_{m,n}^\mathrm{I} (r)\ y_{n,j}^{(i)}(\xi) 
\coloneqq \sqrt{\frac{4m+2n+3}{\beta^3}} P_m^{(0,n+1/2)} \left( \frac{2r^2}{\beta^2} - 1 \right) \left( \frac{r}{\beta} \right)^n y_{n,j}^{(i)}(\xi). 
\intertext{The systems $\mathrm{II}$ and $\mathrm{III}$ are defined for $x \in \ball \backslash \{ 0 \}$. The system $\mathrm{II}$ is given by} 
g_{m,n,j}^{(\mathrm{II},i)}\ (r\xi) 
&\coloneqq F_{m,n}^\mathrm{II} (r)\ y_{n,j}^{(i)}(\xi) 
\coloneqq \sqrt{\frac{2m+3}{\beta^3}} P_m^{(0,2)}\left( \frac{2r}{\beta} - 1 \right) y_{n,j}^{(i)}(\xi)
\intertext{and the system $\mathrm{III}$ is defined by}
g_{m,n,j}^{(\mathrm{III},i)}\ (r\xi)  
&\coloneqq F_{m,n}^\mathrm{III} (r)\ y_{n,j}^{(i)}(\xi) 
\coloneqq \sqrt{\frac{4m+2n+1}{\beta^3}} P_m^{(0,n-1/2)}\left( \frac{2r^2}{\beta^2} - 1 \right) \left( \frac{r}{\beta} \right)^{n-1} y_{n,j}^{(i)}(\xi).
\end{align*}
\end{defi}
\enlargethispage{.5cm}
Note that only system I is well-defined in the origin. However, in the sense of $\mathrm{l}^2(\ball)$, this can be neglected. Further, the systems I and III are obviously very similar.
Wherever in this paper statements are made that hold true for every system I, II and III, the formulation $g_{m,n,j}^{(\star,i)}$ will be used. 

\begin{theo}
By construction, the functions of each system given in \cref{gmnj} are orthonormal and complete in $\mathrm{l}^2(\ball)$.
\end{theo}

\begin{proof} 
The inner product of $\mathrm{l}^2(\ball)$ of $g_{m,n,j}^{(\star,i)}$ and $g_{m',n',j'}^{(\star, i')}$ yields
\begin{align}
\label{tmp100}
\left\langle g_{m,n,j}^{(\star, i)}, g_{m',n',j'}^{(\star,i')} \right\rangle_{\mathrm{l}^2(\ball)} 
&= \int_{\ball}^{} g_{m,n,j}^{(\star,i)}(x) \cdot g_{m',n',j'}^{(\star,i')}(x)\ \mathrm{d}x \notag\\
&= \int_{0}^{\beta}\int_{\Omega}^{} g_{m,n,j}^{(\star,i)}(r\xi) \cdot g_{m',n',j'}^{(\star,i)}(r\xi)\ r^2\ \mathrm{d}\omega(\xi)\ \mathrm{d}r \notag \\
&= \int_{0}^{\beta}F_{m,n}^\star(r) F_{m',n'}^\star (r) r^2\ \mathrm{d}r \int_{\Omega}^{} y_{n,j}^{(i)}(\xi) \cdot y_{n',j'}^{(i')}(\xi)\ \mathrm{d}\omega(\xi).
\end{align}
The orthonormality and completeness of the vector spherical harmonics on $\mathrm{l}^2(\Omega)$ are well known, see for instance \cite[pp.~219-220]{FreedenGutting2013}. The integral over  $[0, \beta]$ is discussed, for example, in \cite[pp.~249--252]{Michel2013} for systems I and II, and in \cite[Th.~3.1]{MichelOrzlowski2016} for system III. Both references include the orthonormality and completeness of the radial systems $F_{m,n}^\star,\ \star \in \{ \mathrm{I}, \mathrm{II}, \mathrm{III} \}$ on the corresponding weighted $\mathrm{L}^2$-space on $[0, \beta]$. Hence, the combination of these systems forms an orthonormal basis in the space $\mathrm{l}^2(\ball)$. 
\end{proof}

The systems from \cref{gmnj} serve in a Fourier expansion of a function $f \in \mathrm{l}^2(\ball)$:
\begin{align*}
f = \sum_{i=1}^{3} \sum_{m=0}^{\infty} \sum_{n=0_i}^{\infty} \sum_{j=-n}^{n} f^\star_{i,m,n,j} g_{m,n,j}^{(\star,i)}, \quad
f^\star_{i,m,n,j} \coloneqq \int_{\ball}^{}f(x)\cdot g_{m,n,j}^{(\star,i)}(x)\ \mathrm{d}x.
\end{align*} Note that this equality includes that the series above converges in the sense of $\mathrm{l}^2(\ball)$.

In this paper, summations of Fourier expansions of vectorial functions will be written as
\begin{align*}
\sum_{i,m,n,j}^{} := \sum_{i=1}^{3} \sum_{m=0}^{\infty} \sum_{n=0_i}^{\infty} \sum_{j=-n}^{n}, \quad
\sum_{i,m,n,j}^{M,N} := \sum_{i=1}^{3} \sum_{m=0}^{M} \sum_{n=0_i}^{N} \sum_{j=-n}^{n}. 
\end{align*}
Analogous versions are used if one or more summations are missing. Furthermore, the vector of all Fourier coefficients is defined as $$\hat{f}^\star := (f^\star_{i,m,n,j})_{i=1,2,3,m \geq 0,n \geq 0_i, j=-n,...,n}$$ for each system I, II and III, where $\star$ again represents this choice.

\section{Localisation of bandlimited vector fields} \label{Sleps}

Now we can formulate the localisation problem. For a certain subregion, we give the entries of the Gramian matrix. At last, we consider the number of well-localised functions for this subregion. From now on, every vector field $f$ is bandlimited. This means we have
\begin{align*} 
f = \sum_{i,m,n,j}^{M,N} f^\star_{i,m,n,j}\ g_{m,n,j}^{(\star,i)}
\end{align*}
for $M \in \uint_0, N \in \uint_{0_i}$ and $\star \in \{ \mathrm{I}, \mathrm{II}, \mathrm{III}\}$.

\subsection{Mathematical formulation with respect to arbitrary localisation regions} \label{Sleps_1}
The Slepian functions shall define a basis of a finite-dimensional subspace of $\mathrm{l}^2(\ball)$. However, an everywhere vanishing vector field cannot be a basis function. Therefore, we assume that $f \not \equiv 0$.

Further, the Slepian functions shall be localised in a measurable subset $R \subseteq \ball$. Therefore, a measure for the localisation of a vector field has to be defined, which will be done here in accordance with the known concept of Slepian functions for the other cases.

\begin{defi}
For a square-integrable vector field $f \colon \ball \to \real^3$, a localisation parameter is formulated by the energy ratio
\begin{align*}
\lambda \coloneqq \frac{ \| f \|^2_{\mathrm{l}^2(R)}}{\| f \|^2_{\mathrm{l}^2(\ball)}} = \frac{  \displaystyle\int_{R} f(x) \cdot f(x)\ \mathrm{d}x}{  \displaystyle\int_{\ball} f(x) \cdot f(x)\ \mathrm{d}x} \ .
\end{align*}
\end{defi}
It clearly holds true that $\lambda \in [0,1]$. The energy ratio and the expansion yield an eigenvalue problem as follows: 
\begin{align*}
\lambda
= \frac{ \sum\limits_{i,m,n,j}^{M,N} \sum\limits_{i',m',n',j'}^{M,N} f^\star_{i,m,n,j} \left[  \displaystyle\int_{R}^{} g_{m,n,j}^{(\star,i)} (x) \cdot g_{m',n',j'}^{(\star,i')} (x)\ \mathrm{d} x \right] f^\star_{i',m',n',j'} }{\sum\limits_{i,m,n,j}^{M,N} \left( f^\star_{i,m,n,j} \right)^2} 
= \frac{\left( \hat{f}^\star\right)^\mathrm{T} K^\star  \hat{f}^\star}{\left( \hat{f}^\star\right)^\mathrm{T}  \hat{f}^\star},
\end{align*} 
where we used the Parseval identity in the denominator. The localisation matrix $K^\star$ has the form
\begin{align} \label{LocMat}
K^\star &\coloneqq \begin{pmatrix}
P^\star & 0 & 0 \\
0 & B^\star & D^\star \\ 
0 & (D^\star)^\mathrm{T} & C^\star 
\end{pmatrix}
\eqqcolon \begin{pmatrix}
P^\star & 0 \\
0 & Q^\star \\ 
\end{pmatrix},  
\quad 
K^\star_{p, p'} \coloneqq \int_{R}^{} g_{m,n,j}^{(\star,i)} (x) \cdot g_{m',n',j'}^{(\star,i')} (x)\ \mathrm{d}x 
\end{align} 
for $p = (i,m,n,j)$ and $p' = (i',m',n',j')$.
The submatrix $P^\star$ belongs to the case $i=i'=1$, $B^\star$ to $i=i'=2$, $C^\star$ to $i=i'=3$ and $D^\star$ to $i=2$ and $i'=3$. The submatrix $Q^\star$ combines all four block matrices originating from the tangential problem. The cases $i=1$ with $i' \in \left\{ 2,3 \right\}$ and vice versa vanish as the basis functions $g_{m,n,j}^{(\star,1)}$ are pointwise orthonormal to $g_{m,n,j}^{(\star,2)}$ and $g_{m,n,j}^{(\star,3)}$ in the Euclidean sense. This holds true because $g_{m,n,j}^{(\star,1)}$ is a normal field to concentric spheres around 0 and $g_{m,n,j}^{(\star,i)},\ i \in \{2,3\}$ are tangential fields by construction.  
The matrix $K^\star$ is symmetric due to the symmetry of the Euclidean inner product. Thus, the case $i=3$ and $i'=2$ yields the transpose $(D^\star)^\mathrm{T}$ of the case $i=2$ and $i'=3$ as a submatrix.
The number of rows and columns, respectively, of $K^\star$ is $$Z \coloneqq (M+1)(3(N+1)^2-2).$$ The energy ratio $\lambda$ is an eigenvalue of $K^\star$. The Principal Axis Theorem yields that $K^\star$ can be diagonalised. Moreover, $K^\star$ has only real eigenvalues $\lambda_k^\star$ for positive integers $k \leq  Z$. The related eigenvectors $\widehat{f_k}^\star,\ k \leq Z,$ are also real and form an orthonormal basis of $\real^{Z}.$ Furthermore, they contain the Fourier coefficients of the vector field $f$. The associated  eigenfunctions $f_k^\star$ with index  $k = 1, \dots, Z$ are defined as 
\begin{align*}
f_k^\star \coloneqq \sum\limits_{i,m,n,j}^{M,N} \left( \widehat{f_k} \right)^\star_{i,m,n,j}\ g_{m,n,j}^{(\star,i)}.
\end{align*}
They are called vectorial Slepian functions on the ball.

\begin{rema} \label{WellLocalised}
A few properties of vectorial Slepian functions on the ball shall be remarked at this point. 
\begin{enumerate}
\item Their definition shows that a vectorial Slepian function on the ball is a bandlimited vector field, whose Fourier coefficients constitute an eigenvector of the localisation matrix $K^\star$. The corresponding eigenvalue equals the energy ratio of the vectorial Slepian function.
\item The best-localised vectorial Slepian function on the ball solves the optimisation problem $\lambda \longrightarrow \max$. 
\item For a set of vectorial Slepian functions, the maximal eigenvalue can be determined. The vectorial Slepian functions on the ball with respect to an eigenvalue close to this maximal eigenvalue are called well-localised. 
\item For a fixed $\star \in \{\mathrm{I, II, III} \}$, the system of vectorial Slepian functions on the ball $\{ f_k^{\star}\ \colon \ k = 1, \dots, Z\} $ forms an orthonormal system in $\mathrm{l}^2(\ball)$ as well as an orthogonal system in $\mathrm{l}^2(R)$ due to Parseval's identity and the eigenvector property.
\end{enumerate}
\end{rema}

For practical purposes, the predominant question is what the matrix entries of $K^\star$ look like in detail. If these are calculated, the eigenvalues and -vectors of $K^\star$ can be deter\-mined with the use of well-known methods from numerical linear algebra. Thereby, more or less -- depending on the value of the related eigenvalue -- spatially localised vectorial Slepian functions on the ball are obtained. 

\subsection{Specifications for numerical experiments} \label{Sleps_2}

The entries of $K^\star$ cannot be more specified in the general setting. Both the localisation region $R$ and the vector spherical harmonics within the basis functions $g_{m,n,j}^{(\star,i)}$ have to be fixed. 
The basis functions are constructed with the use of fully normalised spherical harmonics for practical purposes. With these, vector spherical harmonics and, hence, basis functions $g_{m,n,j}^{(\star,i)}$  are constituted.

The localisation regions in the brain that are of interest to the imaging of neural currents are defined as follows.

\begin{defi} 
Let $\beta \in \ ]0,\infty[$ be the radius of the ball. Further, let the parameters  $a \in [0,\beta[,\ b \in \ ]a, \beta]$ and $\Theta \in \ ]0, \pi]$ be fixed.
For the construction of vectorial Slepian functions on the ball, the localisation region $R$ is called the original partial cone and is defined as
\begin{align} \label{localRegion} 
R \coloneqq \left\{ x(r, \lon, t) \in \real^3 \colon a \leq r \leq b,\ 0 \leq \lon < 2\pi,\ \cos ( \Theta ) \leq t \leq 1 \right\},
\end{align} where $r, \lon$ and $t$ are spherical coordinates.
\end{defi}

\begin{wrapfigure}[14]{l}{5cm}
\subfigure{\includegraphics[width=0.3\textwidth]{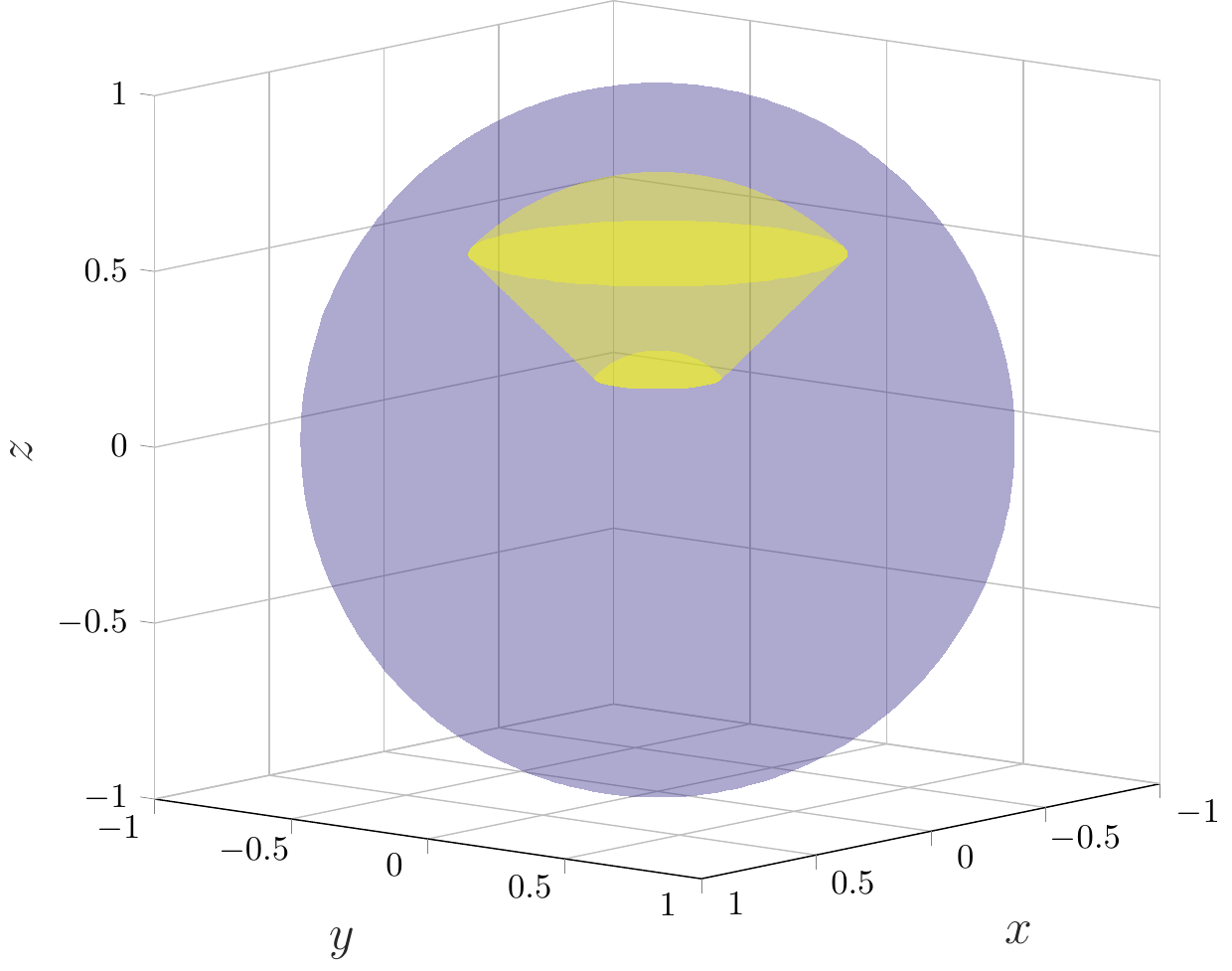}}
\caption[Example of an original partial cone.]{Example of an original partial cone. The boundaries of the cone are marked in yellow. The cone  is defined by $a=0.25,\ b=0.75$ and $\Theta = 45^\circ = \frac{\pi}{4}$. }
\label{OPC}
\end{wrapfigure}

This region resembles a part of a cone. Its height is directed at $\left( 0,0,1 \right)^\mathrm{T}$. Clearly, for a fixed radius $r$, the intersection of $R$ and the sphere with radius $r$ and centre $0$ is a so-called spherical cap. An example of a region of this type is given in \cref{OPC}. With the use of Wigner rotation matrices, as described, for example, in \cite[App.~C.8]{Tromp1998}, the Slepian vector fields can be rearranged. This allows the concentration of functions to partial cones of the type $R$, where, however, the symmetry axis $(t=1)$ is arbitrarily rotated. All cones as defined above have in common that their apex is situated in the origin. 

In combination with the fully normalised spherical harmonics, these subsets of the ball prove to be very convenient. Regarding the entries of $K^\star$, properties like the fact that the spherical harmonics are the eigenfunctions of the Beltrami operator, the periodicity of sine and cosine and Green's surface identities simplify the volume integrals. At this point, the results are assembled. Prior to this overview of matrix entries, certain recurring factors are abbreviated.

\enlargethispage{.5cm}
\begin{defi} \label{MatrixFactors}
A shorthand notation of normalisation factors of Jacobi polynomials and the radial integrals with respect to the systems $\mathrm{I}, \mathrm{II}$ and $\mathrm{III}$ is introduced as follows: 
\begin{align*}
a_{m,m',n,n'}^{\mathrm{I}} &\coloneqq \sqrt{\frac{(4m+2n+3)(4m'+2n'+3)}{2^{n+n'+5}}}\ , \qquad 
a_{m,m',n,n'}^{\mathrm{II}} \coloneqq \sqrt{\frac{(2m+3)(2m'+3)}{64}}\ , \\ 
a_{m,m',n,n'}^{\mathrm{III}} &\coloneqq \sqrt{\frac{(4m+2n+1)(4m'+2n'+1)}{2^{n+n'+3}}}\ ,
\end{align*}
\begin{align*}
I_{m,m',n,n'}^{\mathrm{I}} &\coloneqq \int_{\frac{2a^2}{\beta^2} -1}^{\frac{2b^2}{\beta^2} -1} P_m^{(0,n+1/2)} (u)\ P_{m'}^{(0,n'+1/2)} (u) (u+1)^{(n+n'+1)/2}\ \mathrm{d}u, \\
I_{m,m',n,n'}^{\mathrm{II}} &\coloneqq \int_{\frac{2a}{\beta} -1}^{\frac{2b}{\beta} -1} P_m^{(0,2)} (u)\ P_{m'}^{(0,2)} (u) (u+1)^2\ \mathrm{d}u, \\
I_{m,m',n,n'}^{\mathrm{III}} &\coloneqq \int_{\frac{2a^2}{\beta^2} -1}^{\frac{2b^2}{\beta^2} -1} P_m^{(0,n-1/2)} (u)\ P_{m'}^{(0,n'-1/2)} (u) (u+1)^{(n+n'-1)/2}\ \mathrm{d}u. 
\end{align*}
\end{defi}

This notation provides the matrix entries as follows. The necessary computations can be found in detail in \cref{EntriesApp}.

\begin{theo} \label{MatrixEntries}
Let fully normalised spherical harmonics be used to construct vectorial basis functions of system $\mathrm{I}, \mathrm{II}$ or $\mathrm{III}$. Further, the vectorial Slepian functions on the ball shall be localised in the original partial cone. The localisation matrix of the related eigenvalue problem is denoted by $K^\star$. Then the use of the abbreviations from \cref{MatrixFactors} gives the entries of $K^\star$ with $\star \in \left\{ \mathrm{I}, \mathrm{II}, \mathrm{III} \right\}$ as:
\begin{align*} 
\MoveEqLeft[3] K^\star_{(1,m,n,j),(1,m',n',j')} = a_{m,m',n,n'}^{\star} I_{m,m',n,n'}^{\star} b_{n,j}b_{n',j} \delta_{jj'} \int_{\cos \left(\Theta\right)}^{1} P_{n,|j|}(t) P_{n',|j|}(t)\ \mathrm{d}t, \\
\MoveEqLeft[3] K^\star_{(2,m,n,j),(2,m',n',j')} = K^\star_{(3,m,n,j),(3,m',n',j')} \\
= {} & a_{m,m',n,n'}^{\star} I_{m,m',n,n'}^{\star} b_{n,j}b_{n',j} \delta_{jj'} \left( \sqrt{\frac{n' (n'+1)}{n (n+1)}}\ \int_{\cos \left(\Theta\right)}^{1} P_{n,|j|}(t) P_{n',|j|}(t)\ \mathrm{d}t \right. \notag \\
& \qquad - \left. \frac{\sin^2 \left(\Theta\right)}{\sqrt{n (n+1) n' (n'+1)}}\  P_{n,|j|}(\cos \left(\Theta\right))  P'_{n',|j|}(\cos \left(\Theta\right))\ \vphantom{\sqrt{\frac{n'\ (n'+1)}{n\ (n+1)}}\ \int_{\cos \left(\Theta\right)}^{1}} \right), \\
\MoveEqLeft[3] K^\star_{(2,m,n,j),(3,m',n',j')} \\
= {} & a_{m,m',n,n'}^{\star} I_{m,m',n,n'}^{\star} b_{n,j}b_{n',-j} \frac{j\delta_{-j,j'}}{\sqrt{n (n+1) n' (n'+1)}}\ P_{n,|j|}(\cos \left(\Theta \right))  P_{n',|j'|}(\cos \left(\Theta\right)).
\end{align*}
\end{theo}
The proof of these identities can be found in \cref{EntriesApp}. Note that the submatrices $B^\star$ and $C^\star$ coincide in this setting. 

\subsection{The number of well-localised vector fields} \label{Sleps_3}
The solution of the eigenvalue problem yields as many as $Z$ vectorial Slepian functions on the ball. As mentioned before in \cref{WellLocalised}, these functions can be subdivided into well-localised and poorly-localised ones. Hence, an eigenvalue related to a well-localised function is called significant. Otherwise, it is called insignificant. This `division' helps us to consider the number of well-localised Slepian functions which is called the Shannon number $S^\star$.

The idea of the familiar approach is as follows: the Shannon number can be estimated by the summation of all eigenvalues. On the one hand, significant eigenvalues are values closer to one than to zero. Assume, these values are precisely one. On the other hand, insignificant eigenvalues attain values that are not close to one. In analogy, the assumption is made that the values are strictly zero. Then the summation of all eigenvalues coincides with the number of significant eigenvalues. Hence, it also concurs with the number of well-localised vectorial Slepian functions on the ball. This idea was utilised in the previous works on Slepian functions mentioned above. 

Due to basic linear algebra, see for instance \cite[pp.~229-230]{Fischer2010}, it holds true that similar matrices have the same trace. Hence, the Shannon number can be computed by
\begin{align*}
S^\star = \sum_{k=1}^{Z} \lambda^\star_k = \sum\limits_{i,m,n,j}^{M,N}\ \int_{R}^{} g_{m,n,j}^{(\star,i)}(x) \cdot g_{m,n,j}^{(\star,i)}(x)\ \mathrm{d}x.
\end{align*}
In this paper, the original partial cone is chosen as the localisation region $R$ for practical purposes. This region is a particular case of a general type of subsets of the ball: regions $R$ with independent radial and angular part. For such regions, the volume integrals of the matrix entries can be separated into an integral of the product of two Jacobi polynomials and the integral of the product of two vector spherical harmonics. The first one can be dealt with as in \cref{EntriesApp}. Regarding the latter one, let $ \mathcal{C} $ denote the angular part of this separation. For the summation of the integral over $\mathcal{C}$, we use a derivation from the vectorial addition theorem as seen in \cite[p.~244]{FreedenSchreiner2009}. Thus, if system II is chosen, we obtain
\begin{align*} 
\sum\limits_{i,n,j}^{N}\  \int_{\mathcal{C}}^{} y_{n,j}^{(i)}(\xi) \cdot y_{n,j}^{(i)}(\xi)\  \mathrm{d}\omega(\xi) 
= \frac{3(N+1)^2-2}{4\pi}\ \mathrm{A} (\mathcal{C}), 
\end{align*}
where $\mathrm{A} ( \mathcal{C} )$ indicates the surface area of $\mathcal{C}$. In the case of the original partial cone, this is a spherical cap which can be modelled as a rotation surface. Thus, we obtain 
\begin{align*} 
\mathrm{A} ( \mathcal{C} ) = 2\pi \beta^2(1-\cos \left(\Theta\right))
\end{align*}
for the radius $\beta$ of the ball and the angle $\Theta$ of the spherical cap. If system I or III is chosen, the surface integral can be simplified to
\begin{align*} 
\sum\limits_{i,j}\  \int_{\mathcal{C}}^{} y_{n,j}^{(i)}(\xi) \cdot y_{n,j}^{(i)}(\xi)\  \mathrm{d}\omega(\xi)
= \frac{3(2n+1)}{4\pi}\ \mathrm{A} ( \mathcal{C} ) \qquad \forall \ 0_i \leq n \leq N, 
\end{align*}
with the surface area $\mathrm{A} ( \mathcal{C} )$. With the derivations of the integral of Jacobi polynomials as in \cref{EntriesApp} together with the considerations about the integral of vector spherical harmonics above, the Shannon number rearranges to
\begin{align*} 
S^\mathrm{I} 
&= \beta^2(1-\cos \left(\Theta\right)) \sum\limits_{m,n}^{M,N} \left( \vphantom{\int_{\frac{2a^2}{\beta^2}-1}^{\frac{2b^2}{\beta^2}-1}} \frac{3  (2n+1) (4m+2n+3)}{2^{n+7/2}} \right. \notag \\
& \qquad \qquad \left. \times\int_{\frac{2a^2}{\beta^2}-1}^{\frac{2b^2}{\beta^2}-1} P_m^{(0,n+1/2)} (u)  P_m^{(0,n+1/2)} (u) (u+1)^{n+1/2}\ \mathrm{d}u  \right)
\end{align*} for system I. If system II is selected, the  number is given by
\begin{align*} 
S^\mathrm{II} 
&= \frac{3(N+1)^2-2}{16}\ \beta^2(1-\cos \left(\Theta\right))  \sum\limits_{m}^{M}  (2m+3) \int_{\frac{2a}{\beta}-1}^{\frac{2b}{\beta}-1}   P_m^{(0,2)} (u) P_m^{(0,2)} (u) (u+1)^2\ \mathrm{d}u.  
\end{align*} 
And, in the case of system III,
\begin{align*}
S^\mathrm{III} &=  \beta^2(1-\cos \left(\Theta\right))\sum\limits_{m,n}^{M,N} \left( \vphantom{\int_{\frac{2a^2}{\beta^2}-1}^{\frac{2b^2}{\beta^2}-1}} \frac{3 (2n+1) (4m+2n+1)}{2^{n+5/2}} \right. \notag \\ 
& \qquad \times \left. \int_{\frac{2a^2}{\beta^2}-1}^{\frac{2b^2}{\beta^2}-1} P_m^{(0,n-1/2)} (u)  P_m^{(0,n-1/2)} (u) (u+1)^{n-1/2}\ \mathrm{d}u \right).  
\end{align*} 

Each formula points out a particular property of the vectorial Slepian functions on the ball: the number of well-localised functions depends on the highest possible radial and angular degree as well as on the size of the localisation region. This is in analogy to the previous works on other Slepian functions. Both dependencies can be explained. On the one hand, the more functions are used for the Fourier expansion of a Slepian function, the smaller the differences between two Slepian functions can be. Hence, if the size of $K^\star$ increases, the number of significant eigenfunctions increases as well. On the other hand, the larger the localisation region is, the less the spherical harmonics and Jacobi polynomials have to be suppressed. Spherical harmonics are also polynomials. Hence, both functions are not well-localised. Thus, if the localisation region decreases, it is harder to find well-localised eigenfunctions.\\

\section{Numerical results} \label{Results}

\begin{figure}[!t] 
\centering
\subfigure{\includegraphics[width=0.45\textwidth]{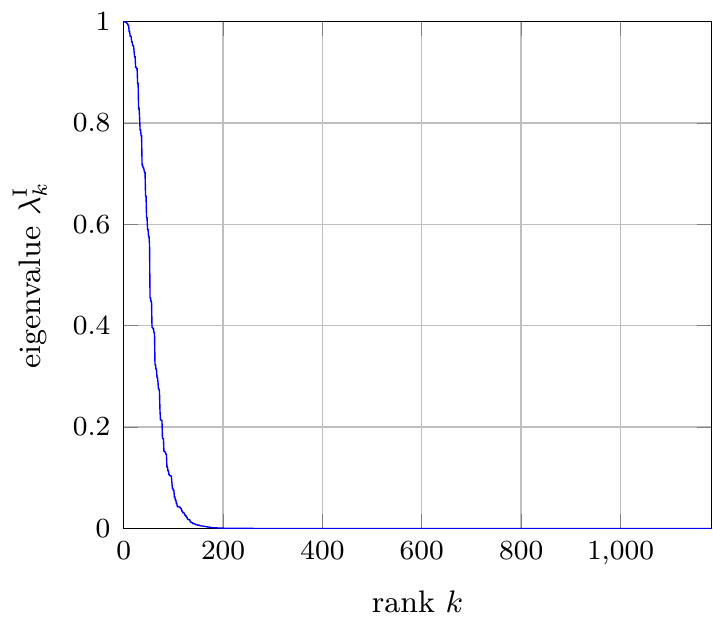}}
\hspace{1cm}
\subfigure{\includegraphics[width=0.45\textwidth]{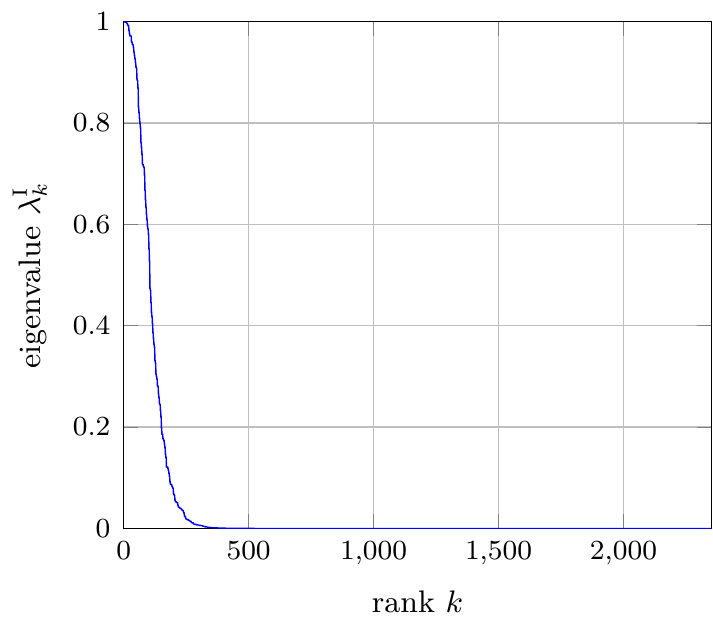}}
\caption{Distribution of the eigenvalues of $P^\mathrm{I}$ (left) and $Q^\mathrm{I}$ (right).} 
\label{EtaSysI}

\subfigure{\includegraphics[width=0.24\textwidth]{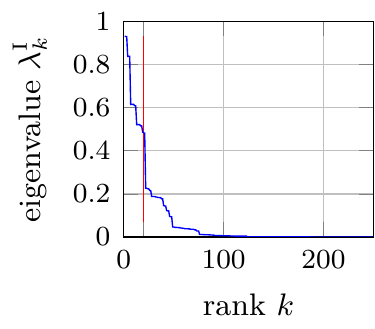}}
\subfigure{\includegraphics[width=0.24\textwidth]{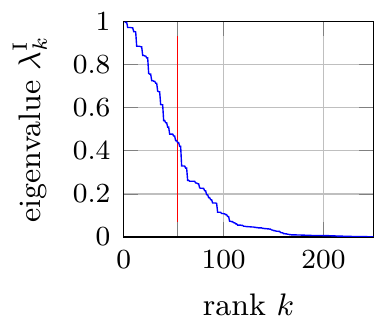}}
\subfigure{\includegraphics[width=0.24\textwidth]{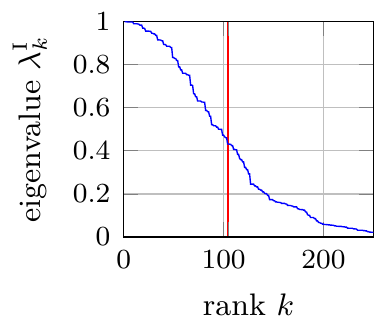}}
\subfigure{\includegraphics[width=0.24\textwidth]{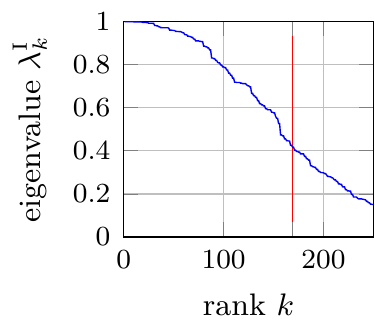}}
\caption{The distribution of the eigenvalues (blue) of $K^\mathrm{I}$  in the cases $\Theta=15^\circ,\ 25^\circ,\ 35^\circ,\ 45^\circ$ (left to right) is shown. With red, the Shannon number is marked. It is given by $S^\mathrm{I} \approx 20,\ 54,\ 104,\ 169$ (left to right).}
\label{EtaSysIShannon}
\end{figure}

\begin{figure}[!t] 
\centering
\subfigure{\includegraphics[width=0.35\textwidth]{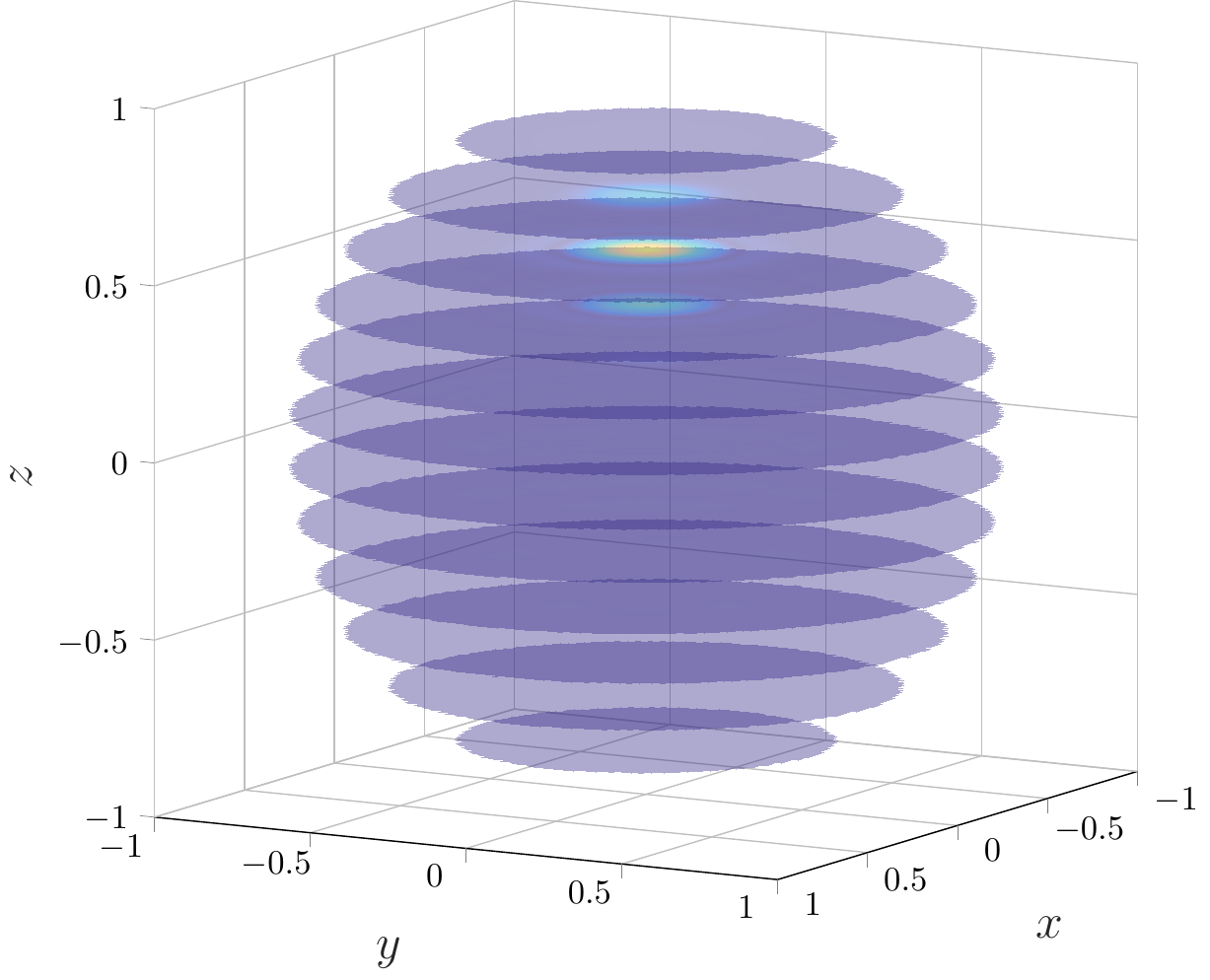}}
\subfigure{\includegraphics[width=0.27\textwidth]{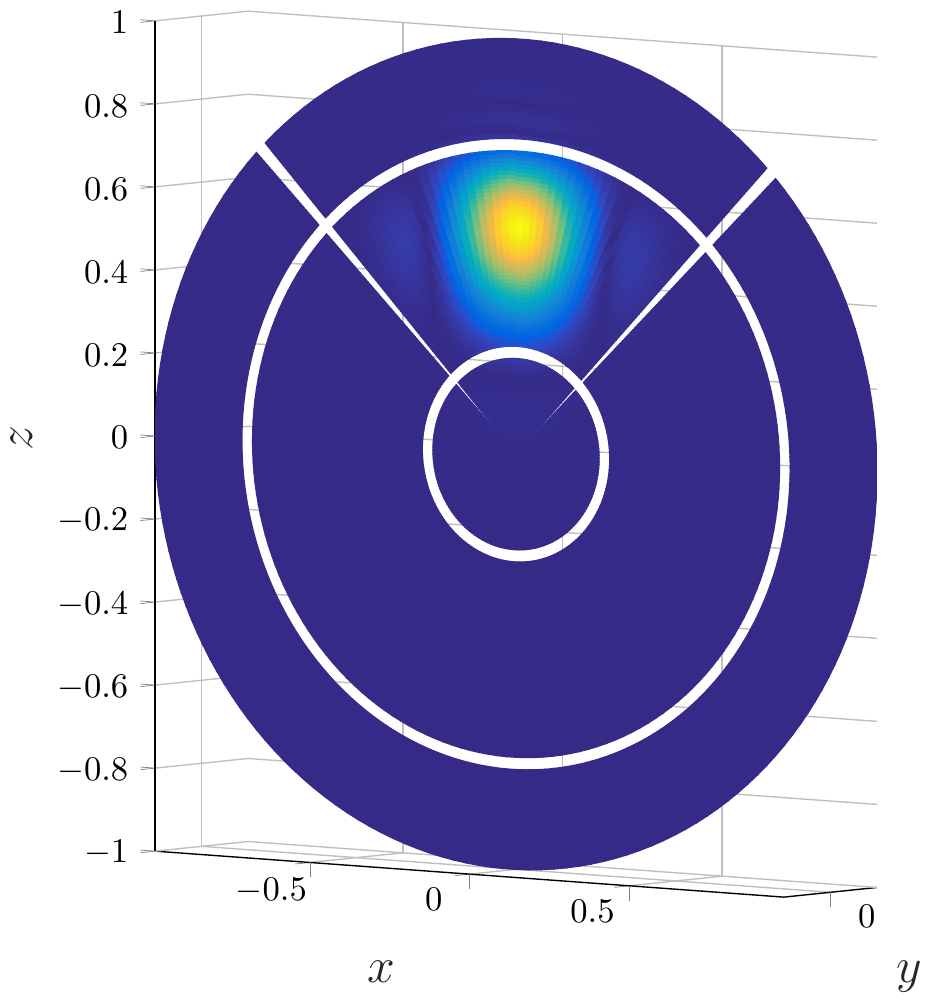}}
\subfigure{\includegraphics[width=0.35\textwidth]{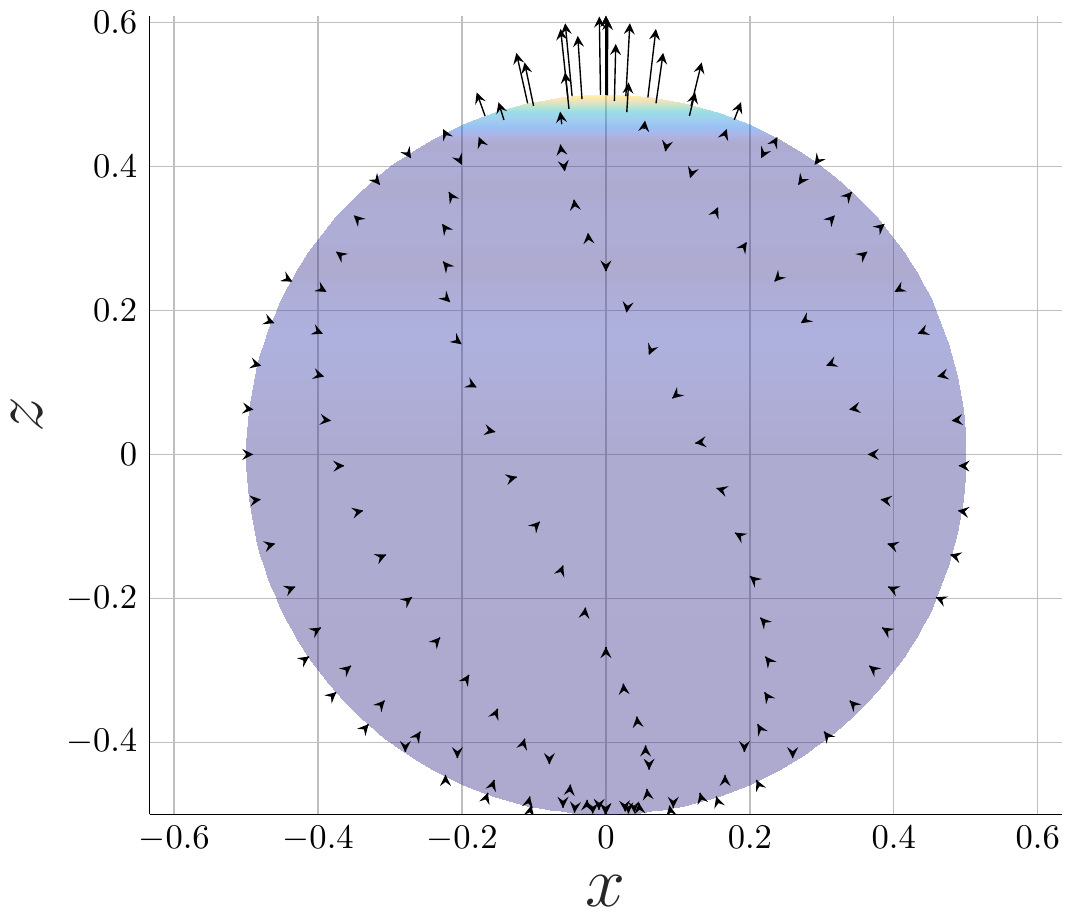}}
\caption{Vectorial Slepian functions on the ball from system I. Here, a normal field with related eigenvalue $0.999056$ is given. The representation of the Euclidean norm in the interior of the ball is shown. Blue depicts values close to zero and yellow stands for large values. Moreover, the vectorial functions are shown on a sphere with radius $0.5$ (right).}
\label{SysIRad}

\subfigure{\includegraphics[width=0.35\textwidth]{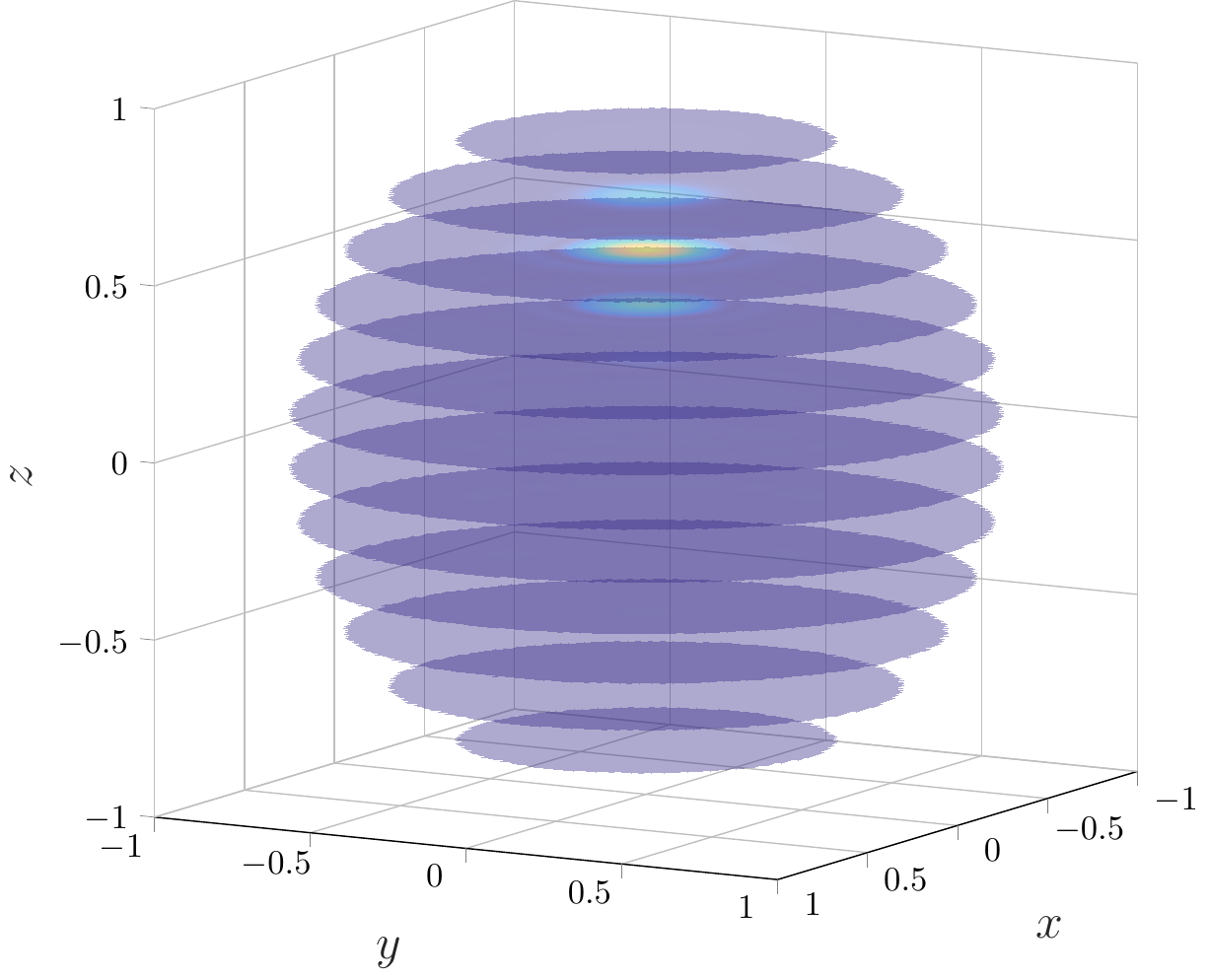}}
\subfigure{\includegraphics[width=0.27\textwidth]{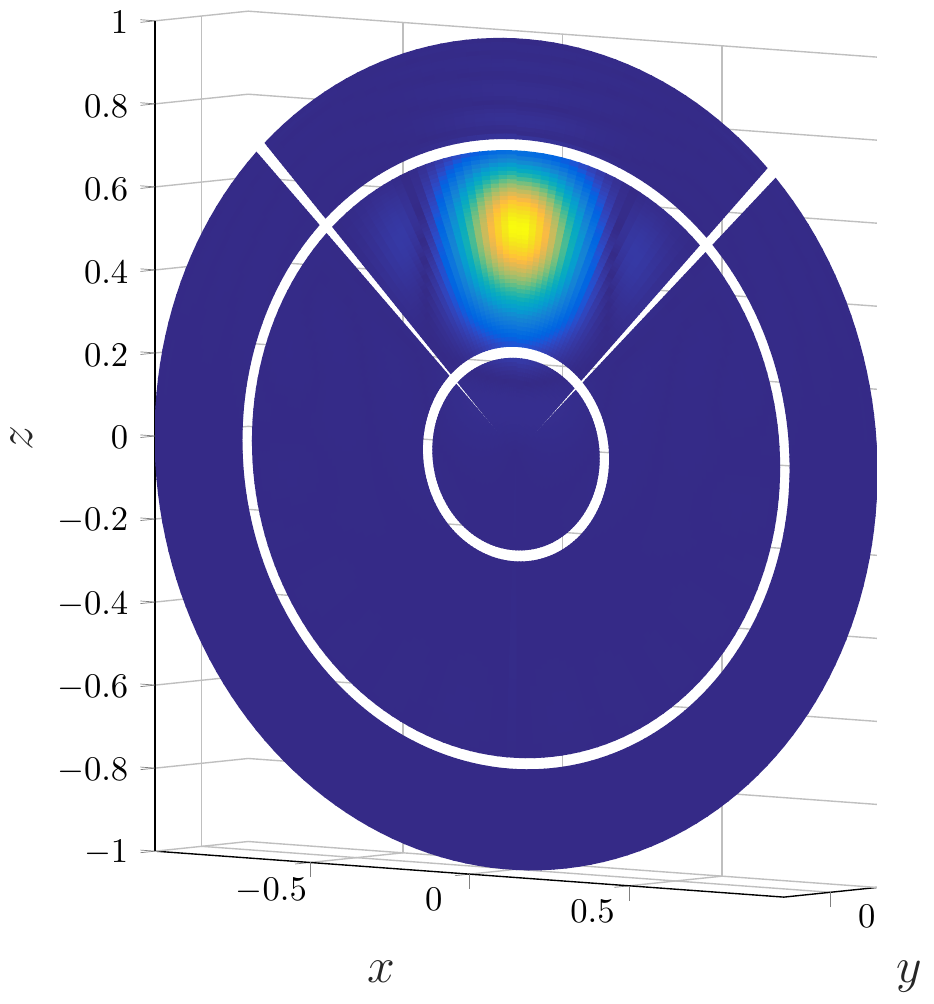}}
\subfigure{\includegraphics[width=0.35\textwidth]{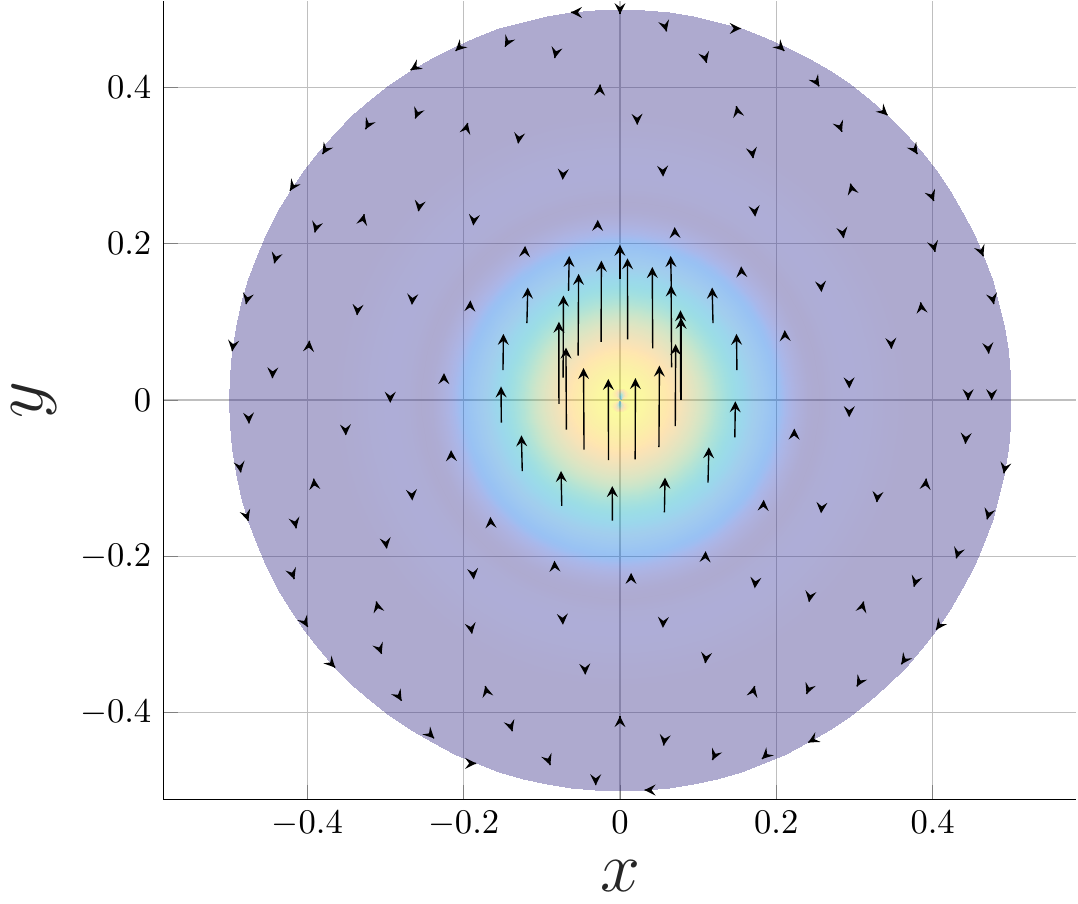}}
\caption{Vectorial Slepian functions on the ball from system I. Here, a tangential field with related eigenvalue $0.999123$ is given. The function is illustrated in a similar manner as in \cref{SysIRad}. Note the rotated coordinate system on the right-hand side (for a better visibility of the vectors).}
\label{SysITan}
\end{figure}

\begin{figure}[!t] 
\centering
\subfigure{\includegraphics[width=0.45\textwidth]{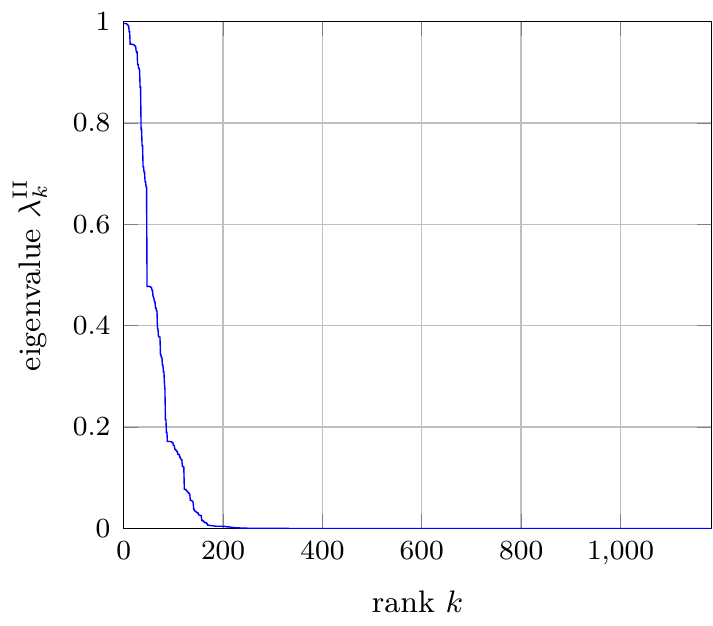}}
\hspace{1cm}
\subfigure{\includegraphics[width=0.45\textwidth]{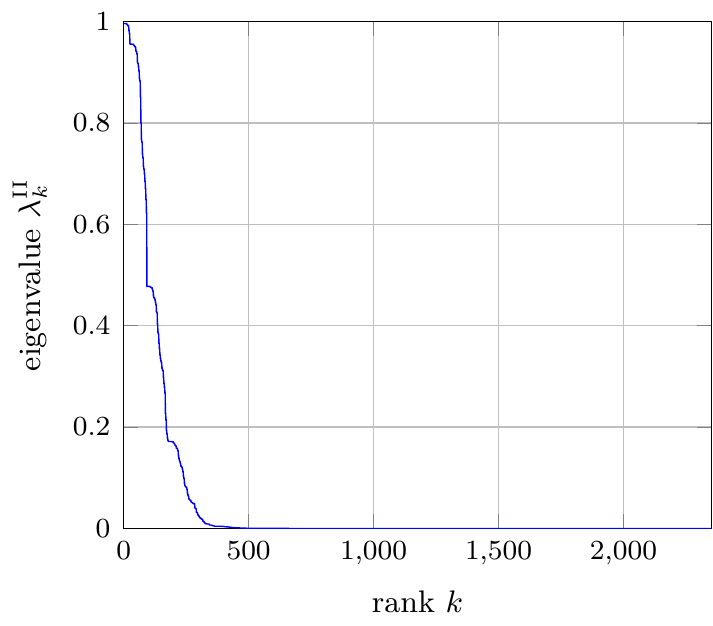}}
\caption{Distribution of the eigenvalues of $P^\mathrm{II}$ (left) and $Q^\mathrm{II}$ (right)-} 
\label{EtaSysII}

\subfigure{\includegraphics[width=0.24\textwidth]{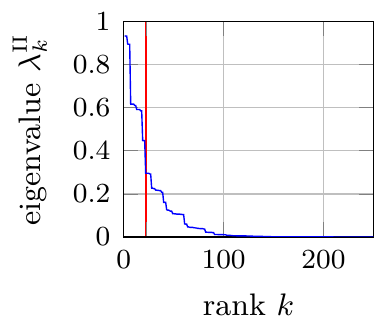}}
\subfigure{\includegraphics[width=0.24\textwidth]{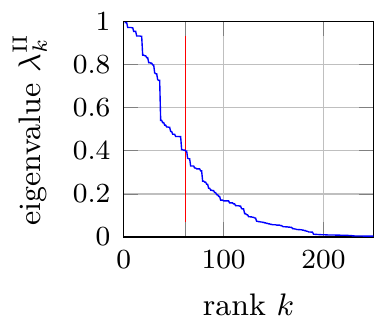}}
\subfigure{\includegraphics[width=0.24\textwidth]{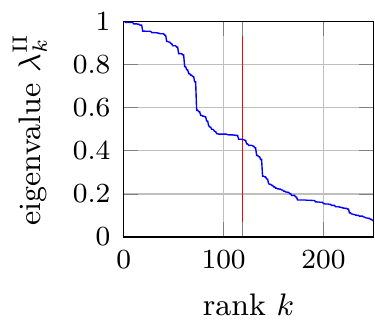}}
\subfigure{\includegraphics[width=0.24\textwidth]{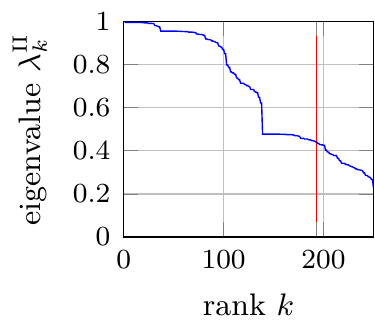}}
\caption{The distribution of the eigenvalues (blue) of $K^\mathrm{II}$  in the cases $\Theta=15^\circ,\ 25^\circ,\ 35^\circ,\ 45^\circ$ (left to right) is shown. With red, the Shannon number is marked. It is given by $S^\mathrm{II} \approx 22,\ 62,\  119,\ 193$ (left to right).}
\label{EtaSysIIShannon}
\end{figure}

\begin{figure}[!t] 
\centering
\subfigure{\includegraphics[width=0.35\textwidth]{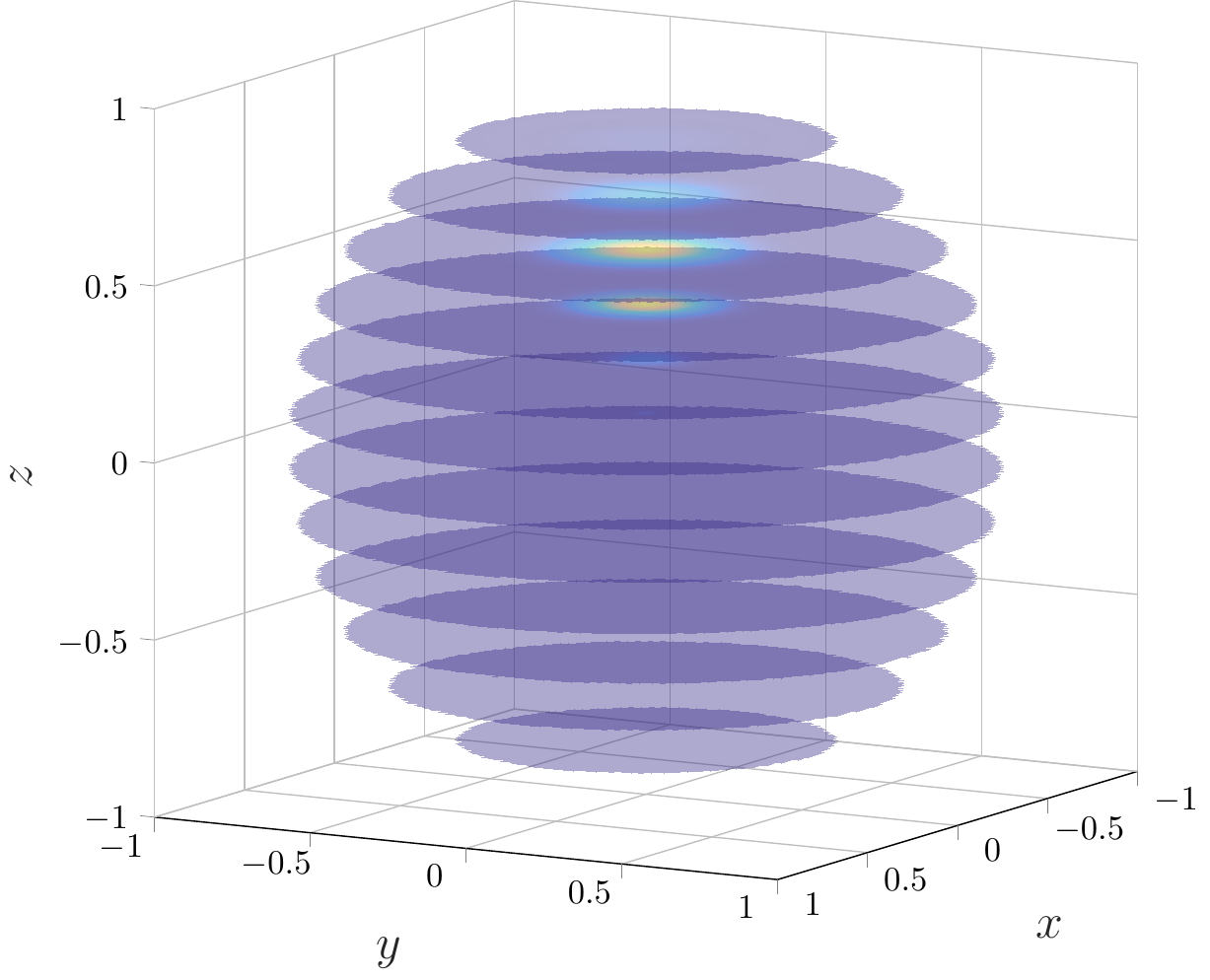}}
\subfigure{\includegraphics[width=0.27\textwidth]{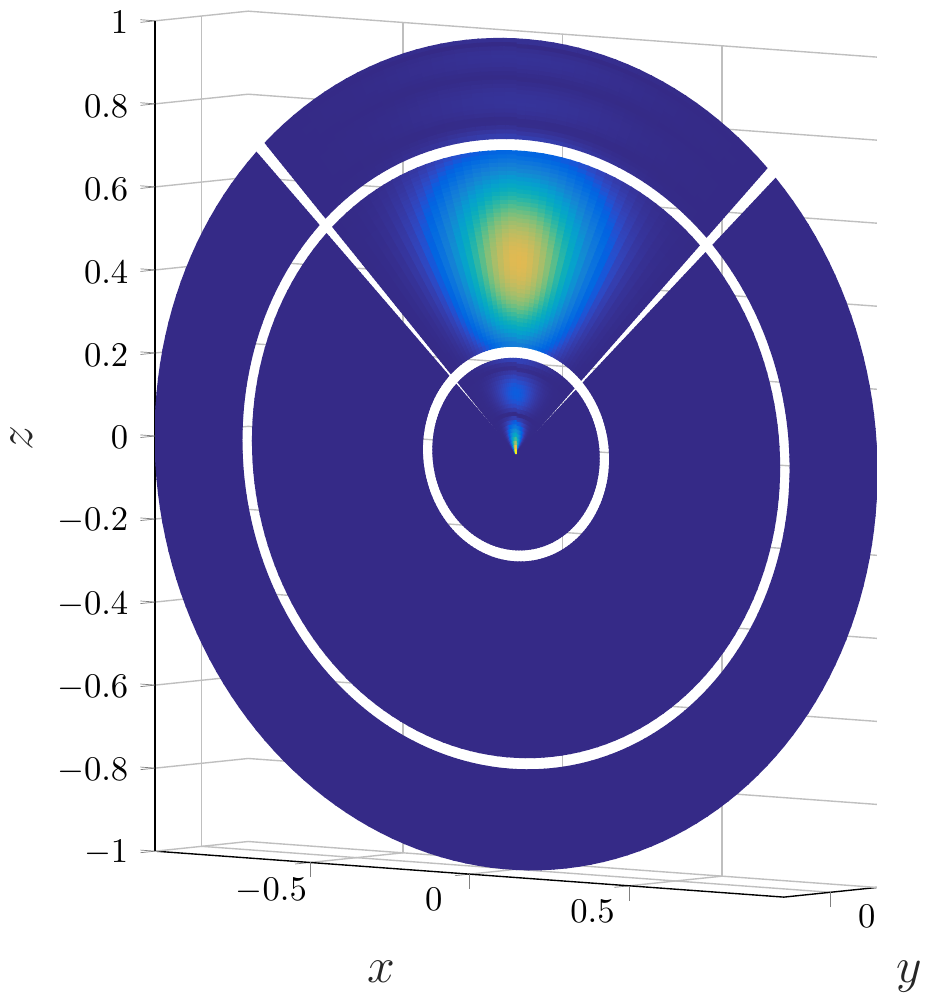}}
\subfigure{\includegraphics[width=0.35\textwidth]{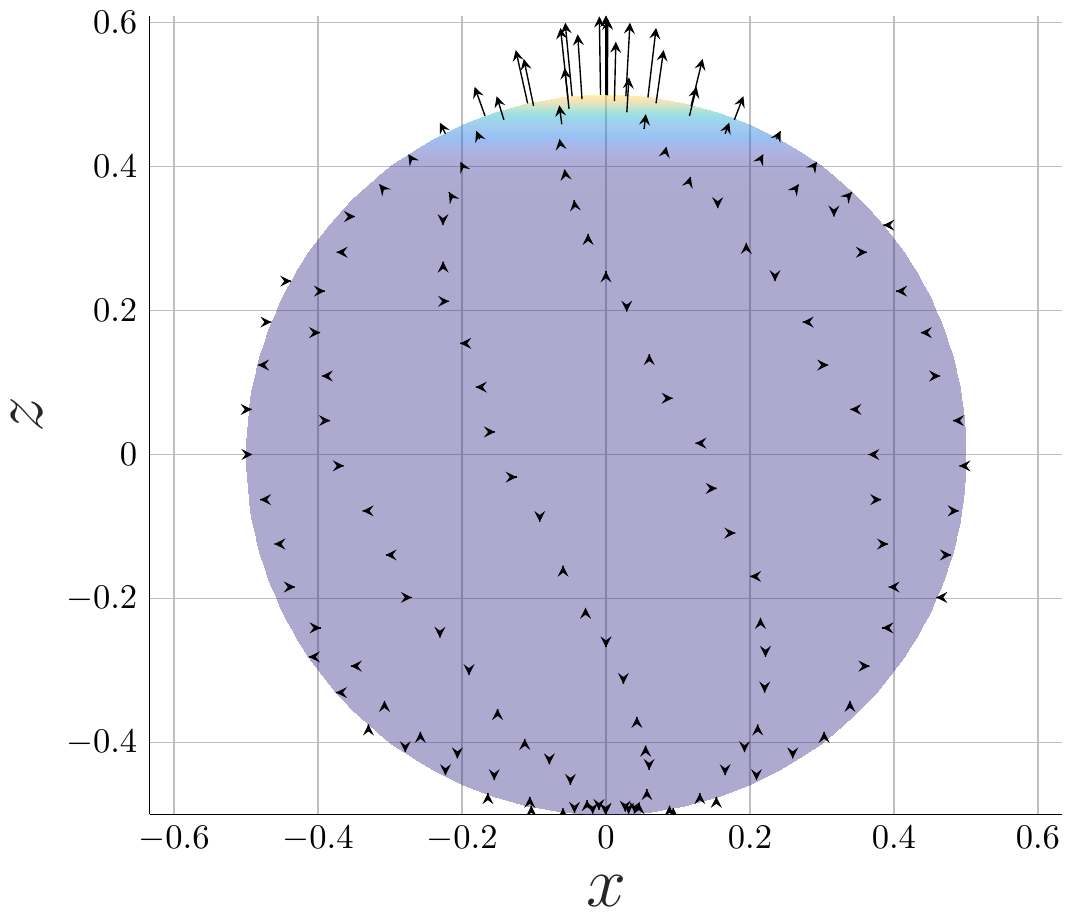}}
\caption{Vectorial Slepian functions on the ball from system II. Here, a normal field with related eigenvalue $0.996101$ is given. The representation of the Euclidean norm in the interior of the ball is shown. Blue depicts values close to zero and yellow stands for large values. Moreover, the vectorial functions are shown on a sphere with radius $0.5$ (right).}
\label{SysIIRad}

\subfigure{\includegraphics[width=0.35\textwidth]{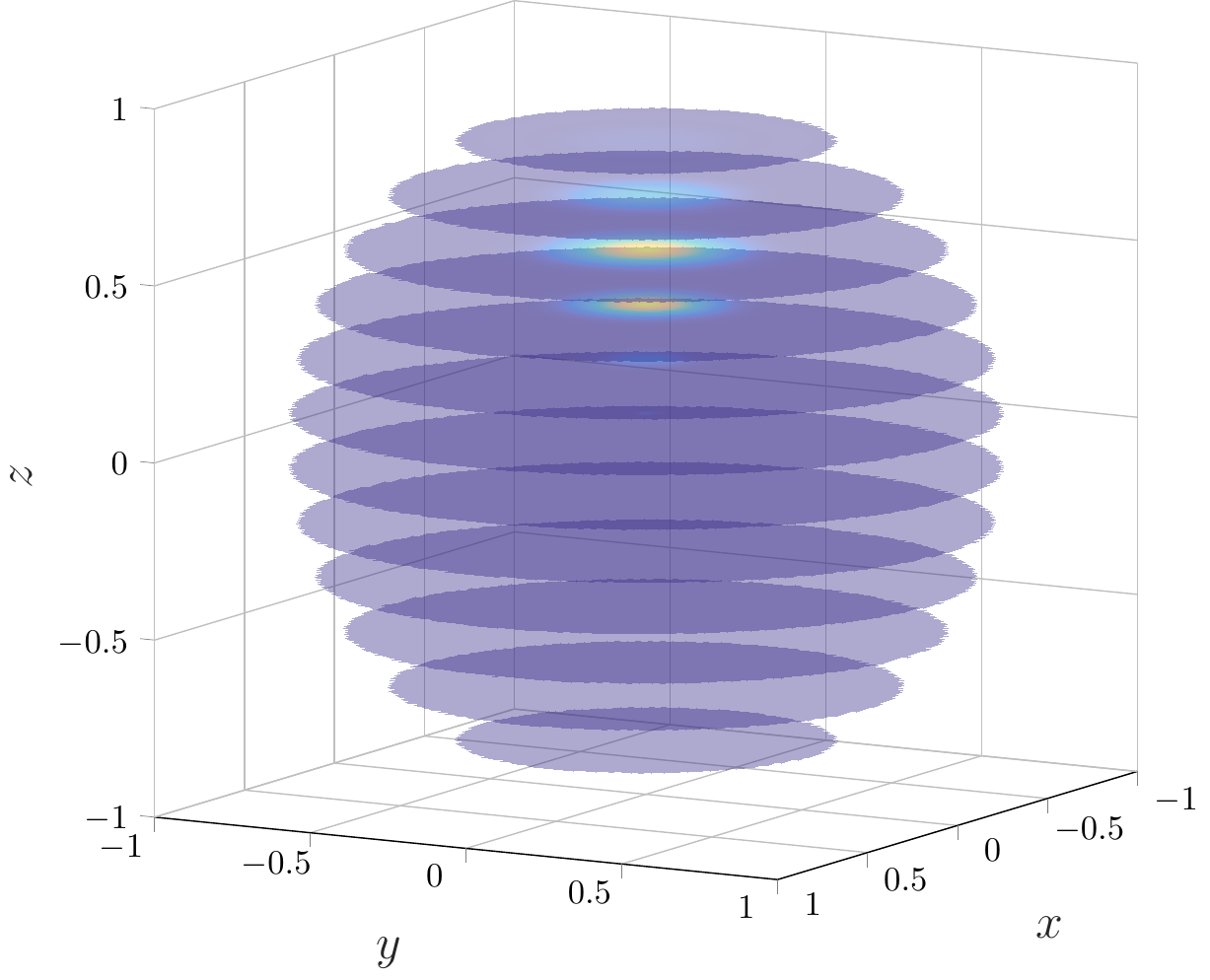}}
\subfigure{\includegraphics[width=0.27\textwidth]{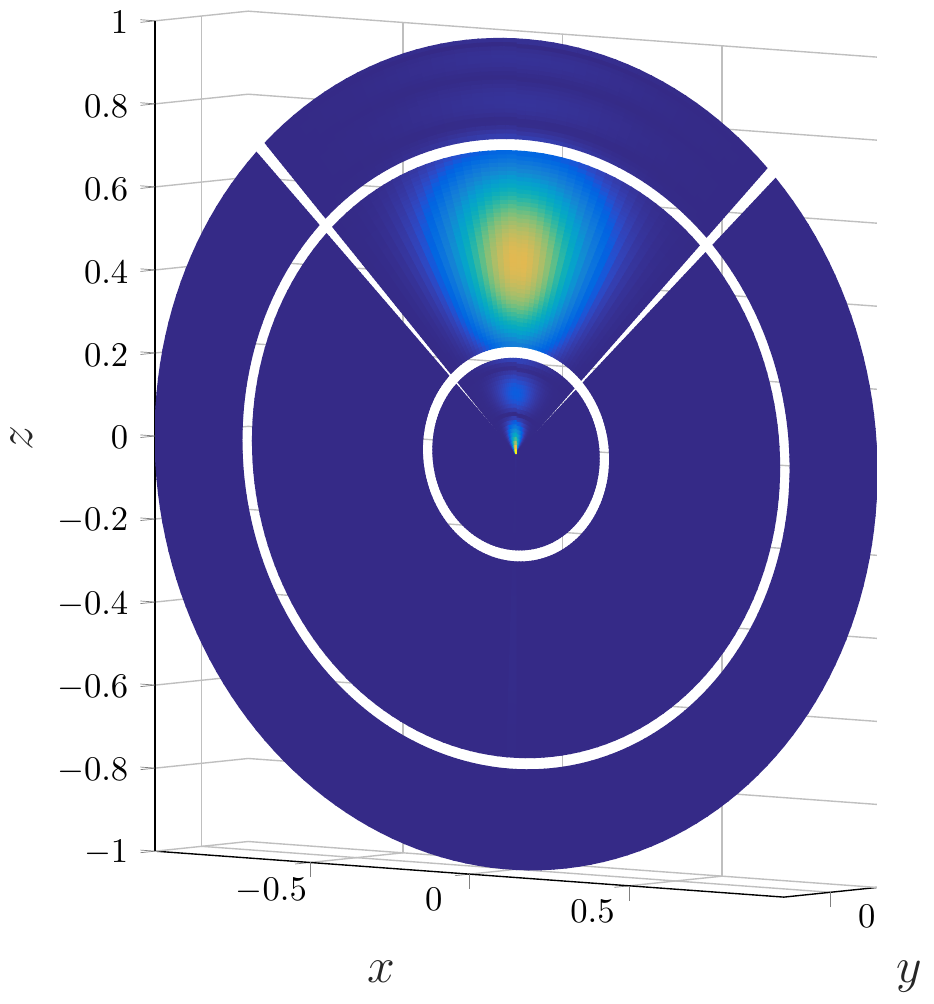}}
\subfigure{\includegraphics[width=0.35\textwidth]{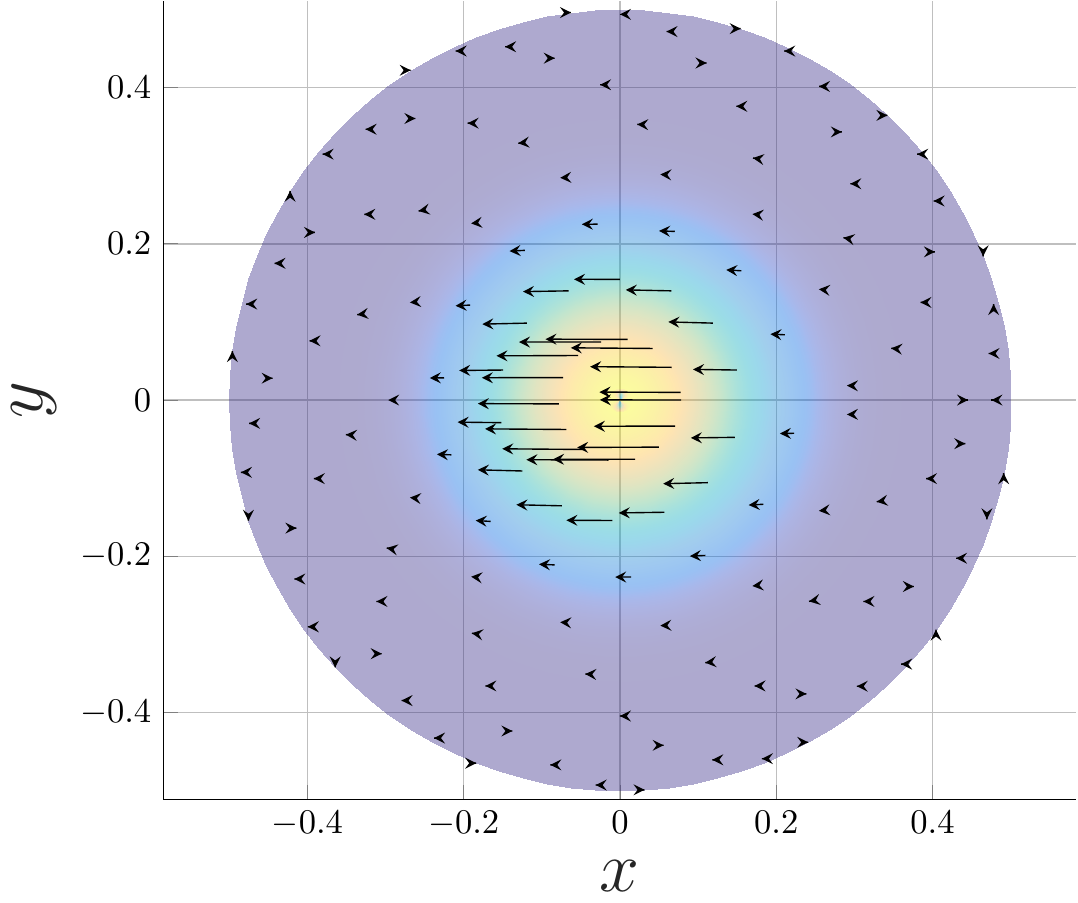}}
\caption{Vectorial Slepian functions on the ball from system II. Here, a tangential field with related eigenvalue $0.996101$ is given. The function is illustrated in a similar manner as in \cref{SysIIRad}. Note the rotated coordinate system on the right-hand side (for a better visibility of the vectors).}
\label{SysIITan}
\end{figure}

\begin{figure}[!t] 
\centering
\subfigure{\includegraphics[width=0.45\textwidth]{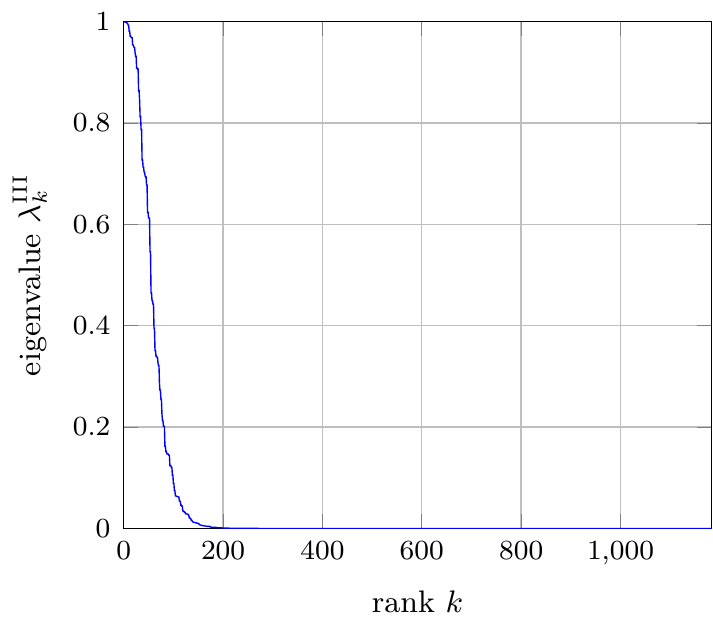}}
\hspace{1cm}
\subfigure{\includegraphics[width=0.45\textwidth]{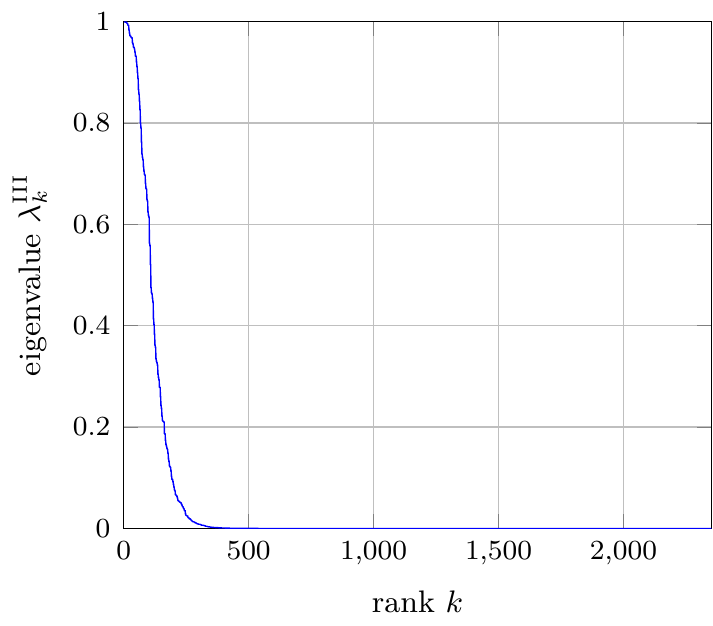}}
\caption{Distribution of the eigenvalues of $P^\mathrm{III}$ (left) and $Q^\mathrm{III}$ (right).} 
\label{EtaSysIII}

\subfigure{\includegraphics[width=0.24\textwidth]{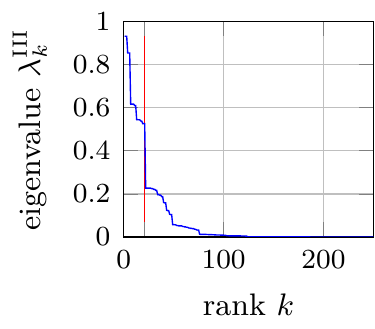}}
\subfigure{\includegraphics[width=0.24\textwidth]{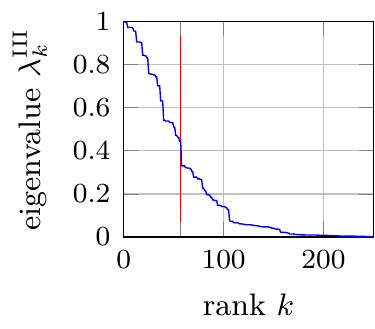}}
\subfigure{\includegraphics[width=0.24\textwidth]{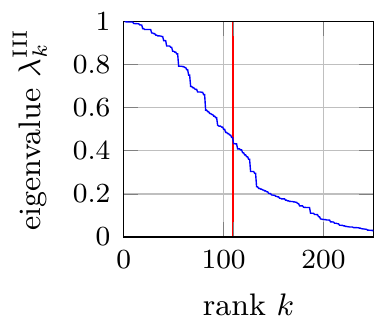}}
\subfigure{\includegraphics[width=0.24\textwidth]{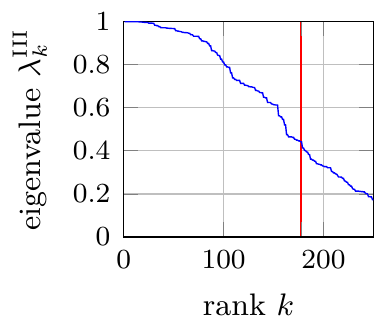}}
\caption{The distribution of the eigenvalues (blue) of $K^\mathrm{III}$  in the cases $\Theta=15^\circ,\ 25^\circ,\ 35^\circ,\ 45^\circ$ (left to right) is shown. With red, the Shannon number is marked. It is given by $S^\mathrm{III} \approx 21,\ 57,\ 109,\ 177$ (left to right).}
\label{EtaSysIIIShannon}
\end{figure}

\begin{figure}[!t] 
\centering
\subfigure{\includegraphics[width=0.35\textwidth]{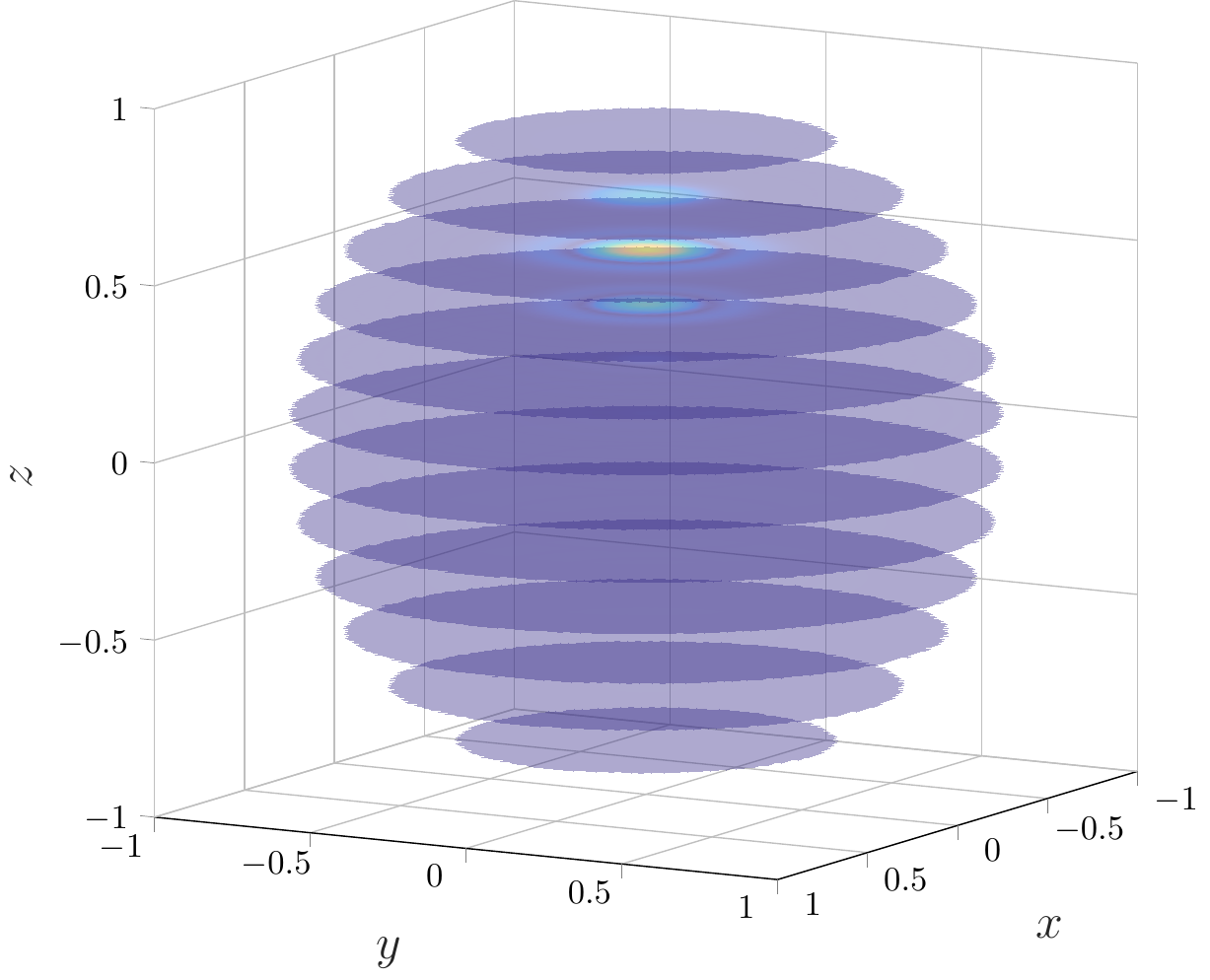}}
\subfigure{\includegraphics[width=0.27\textwidth]{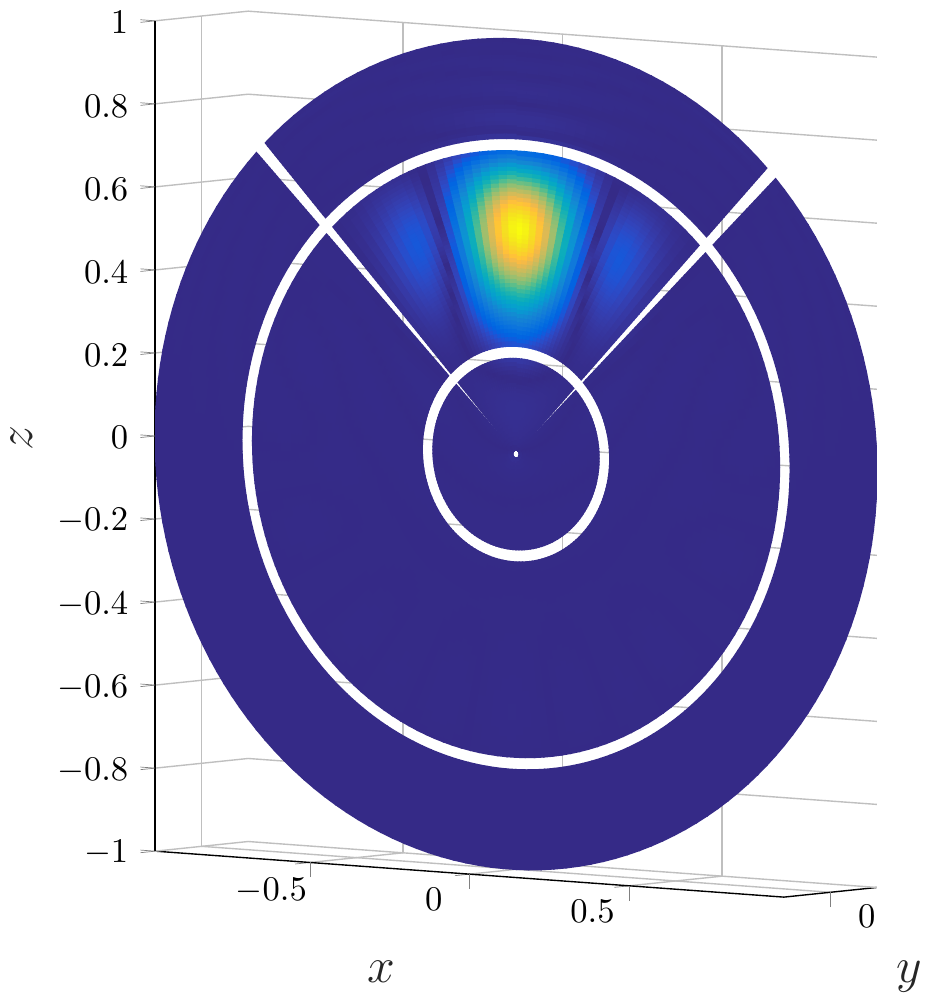}}
\subfigure{\includegraphics[width=0.35\textwidth]{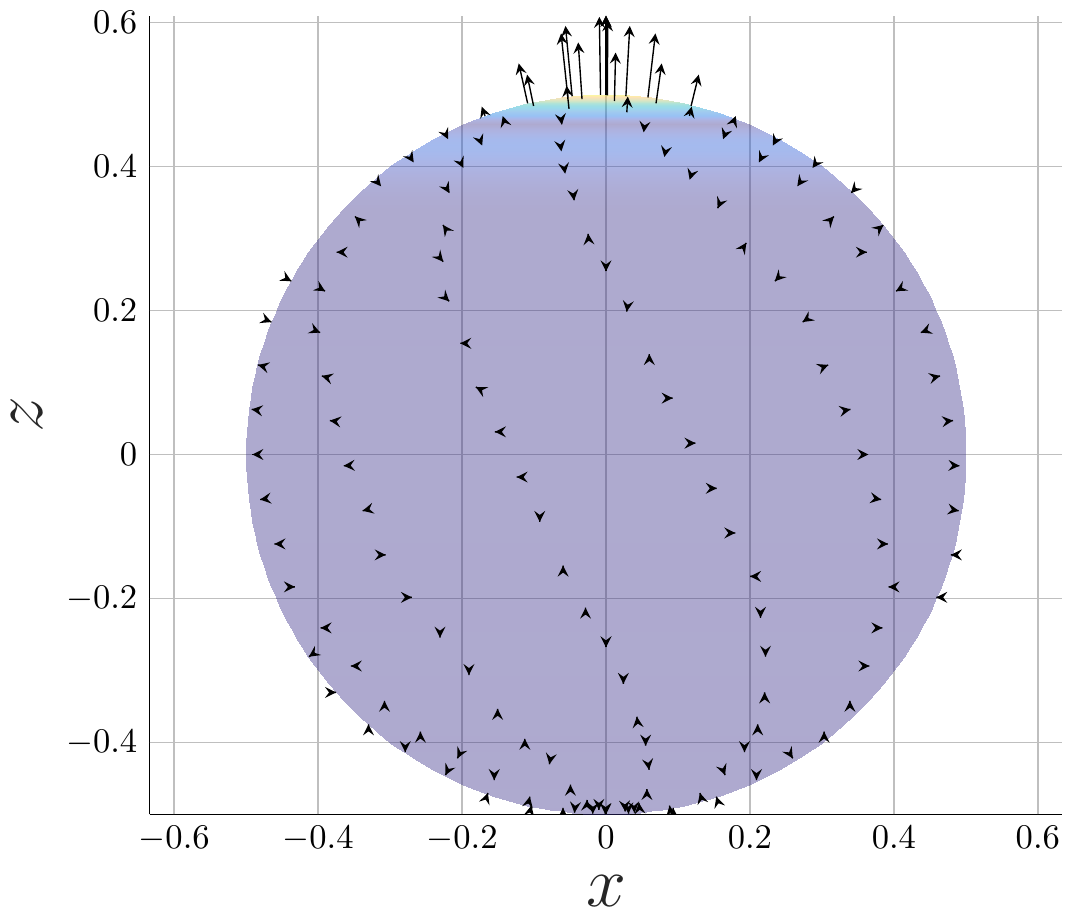}}
\caption{Vectorial Slepian functions on the ball from system III. Here, a normal field with related eigenvalue $0.998982$ is given. The representation of the Euclidean norm in the interior of the ball is shown. Blue depicts values close to zero and yellow stands for large values. Moreover, the vectorial functions are shown on a sphere with radius $0.5$ (right).}
\label{SysIIIRad}

\subfigure{\includegraphics[width=0.35\textwidth]{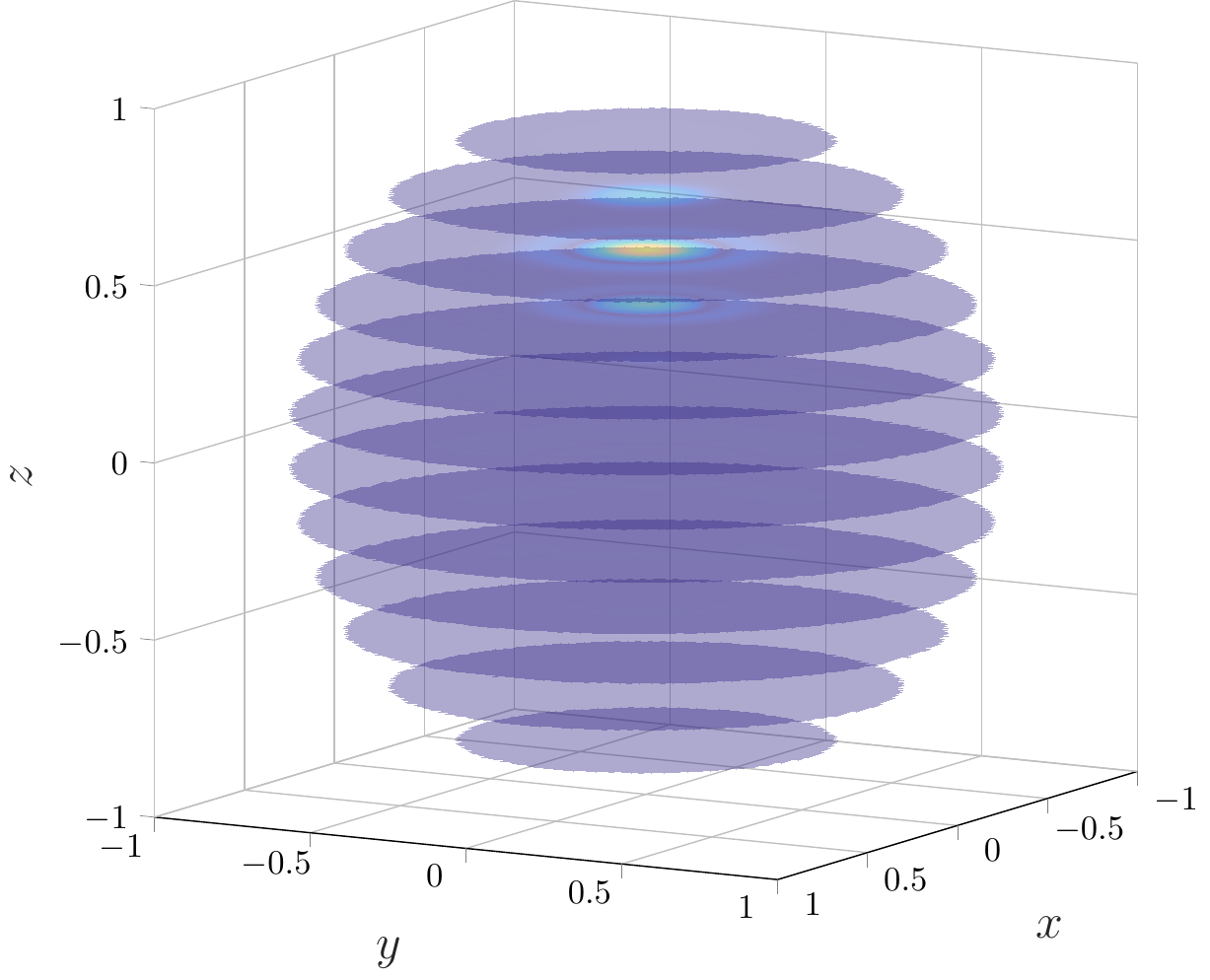}}
\subfigure{\includegraphics[width=0.27\textwidth]{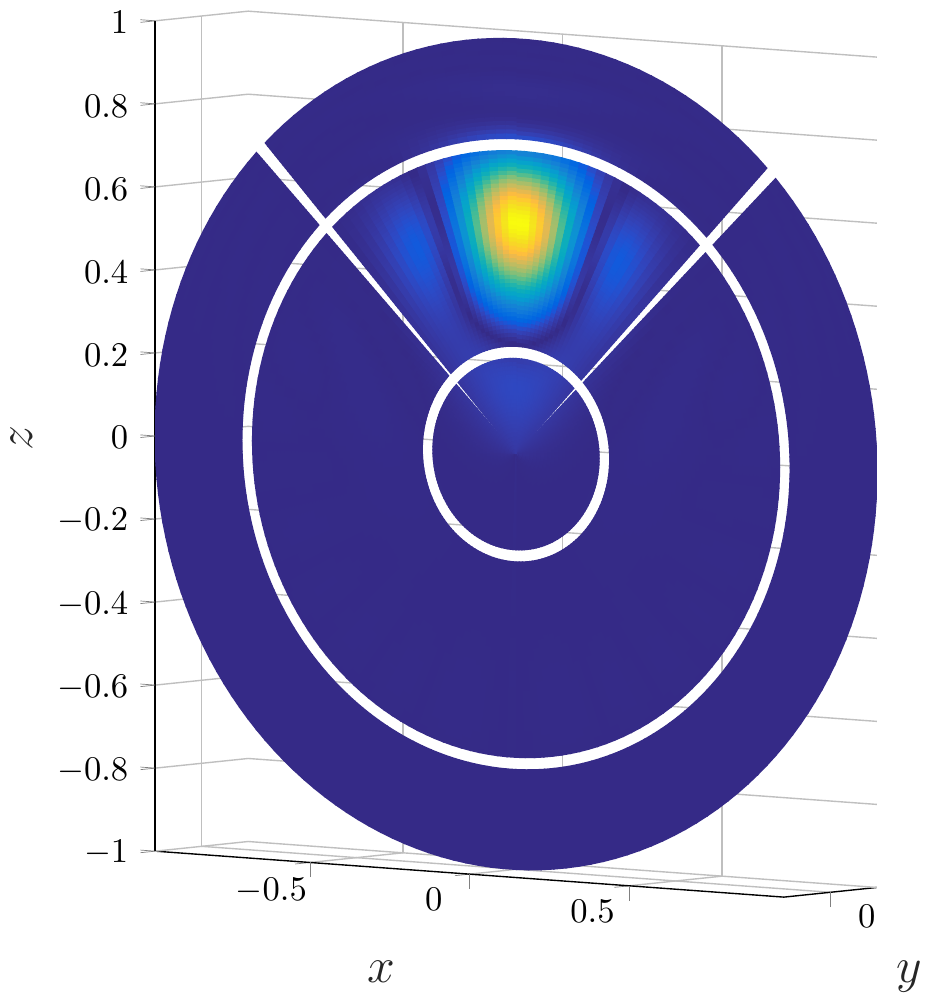}}
\subfigure{\includegraphics[width=0.35\textwidth]{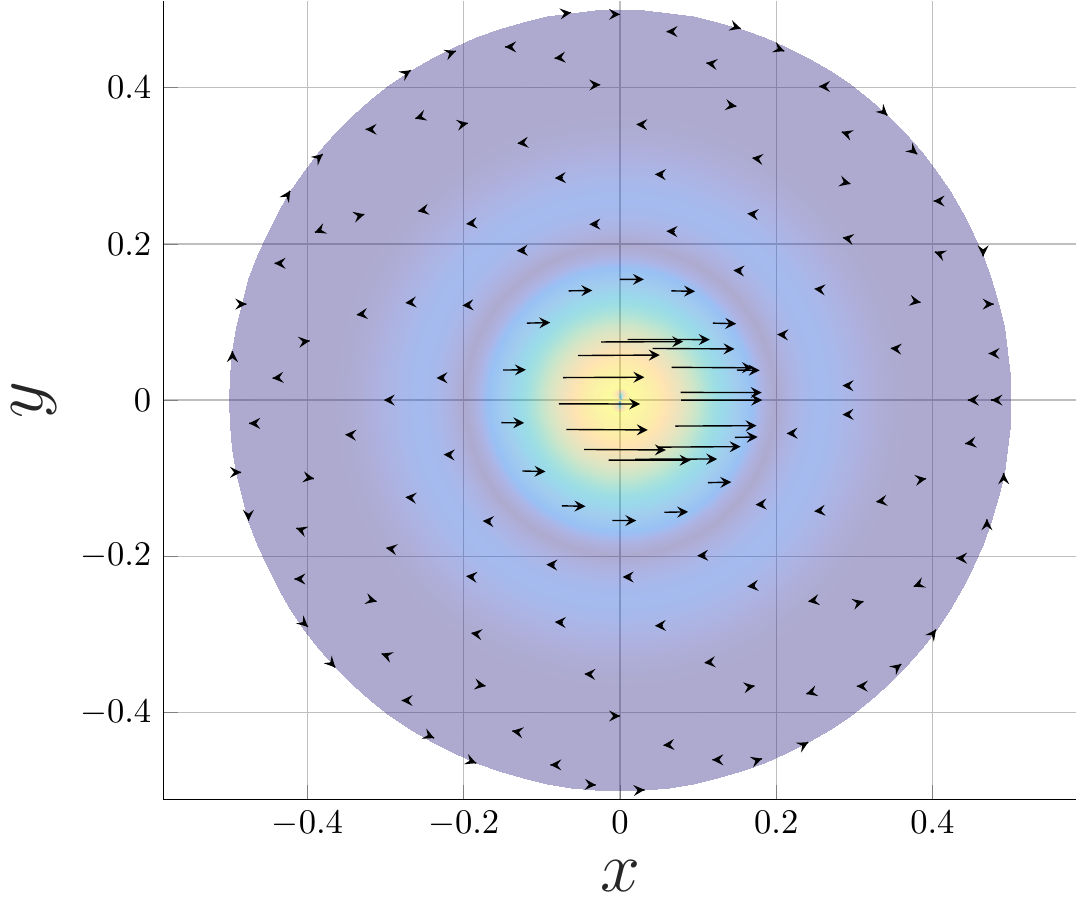}}
\caption{Vectorial Slepian functions on the ball from system III. Here, a tangential field with related eigenvalue $0.998987$ is given. The function is illustrated in a similar manner as in \cref{SysIIIRad}. Note the rotated coordinate system on the right-hand side (for a better visibility of the vectors).}
\label{SysIIITan}
\end{figure}

\begin{figure}[!t] 
\centering
\subfigure{\includegraphics[width=0.35\textwidth]{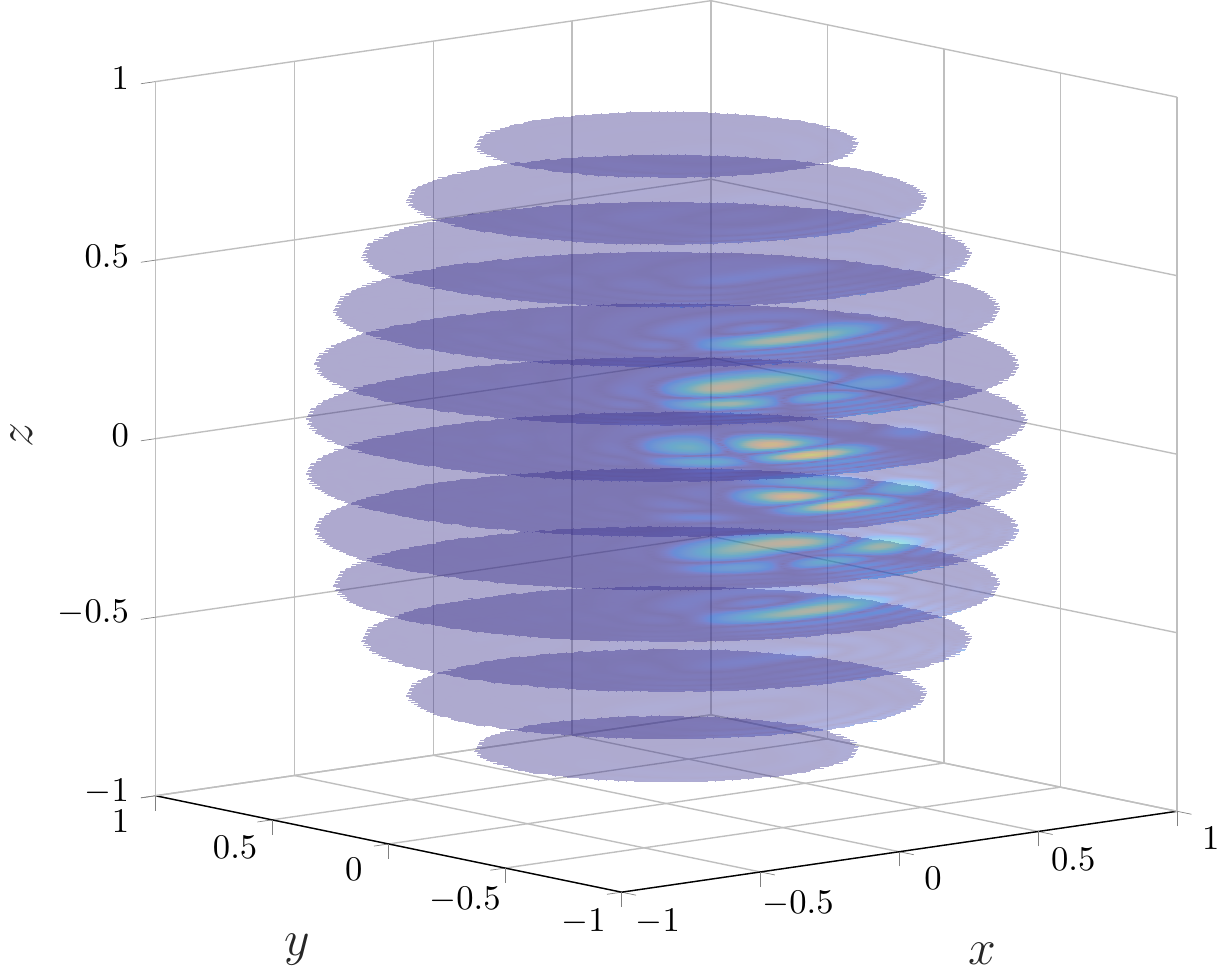}}
\subfigure{\includegraphics[width=0.27\textwidth]{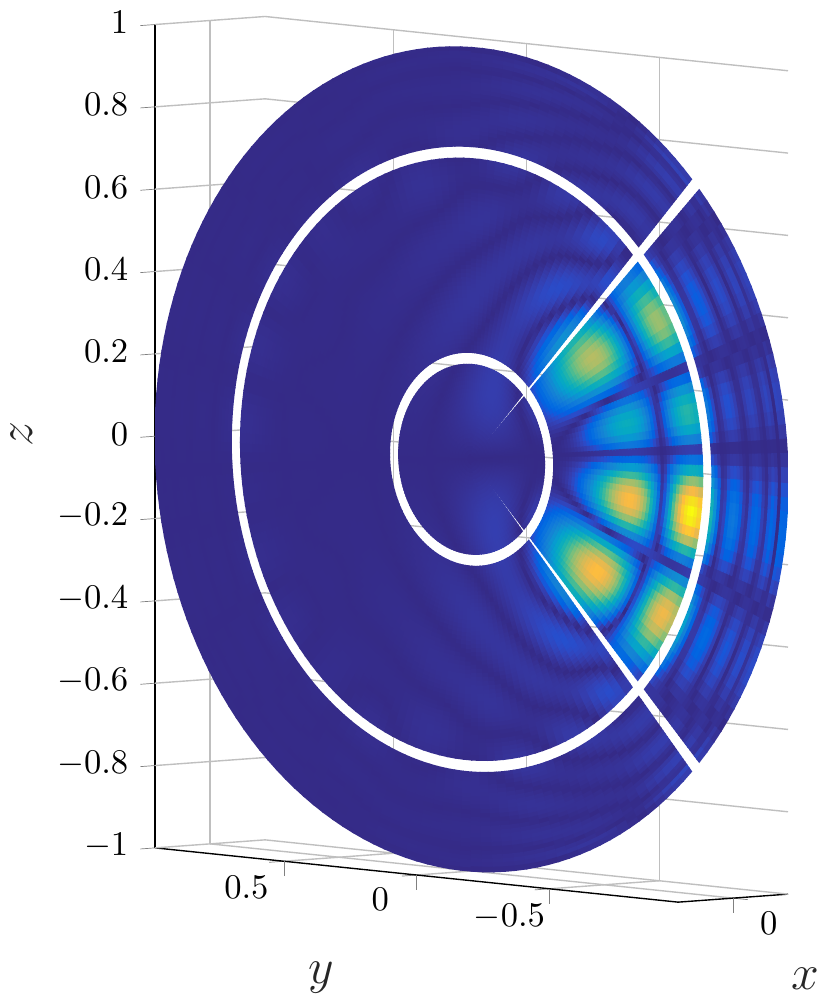}}
\subfigure{\includegraphics[width=0.35\textwidth]{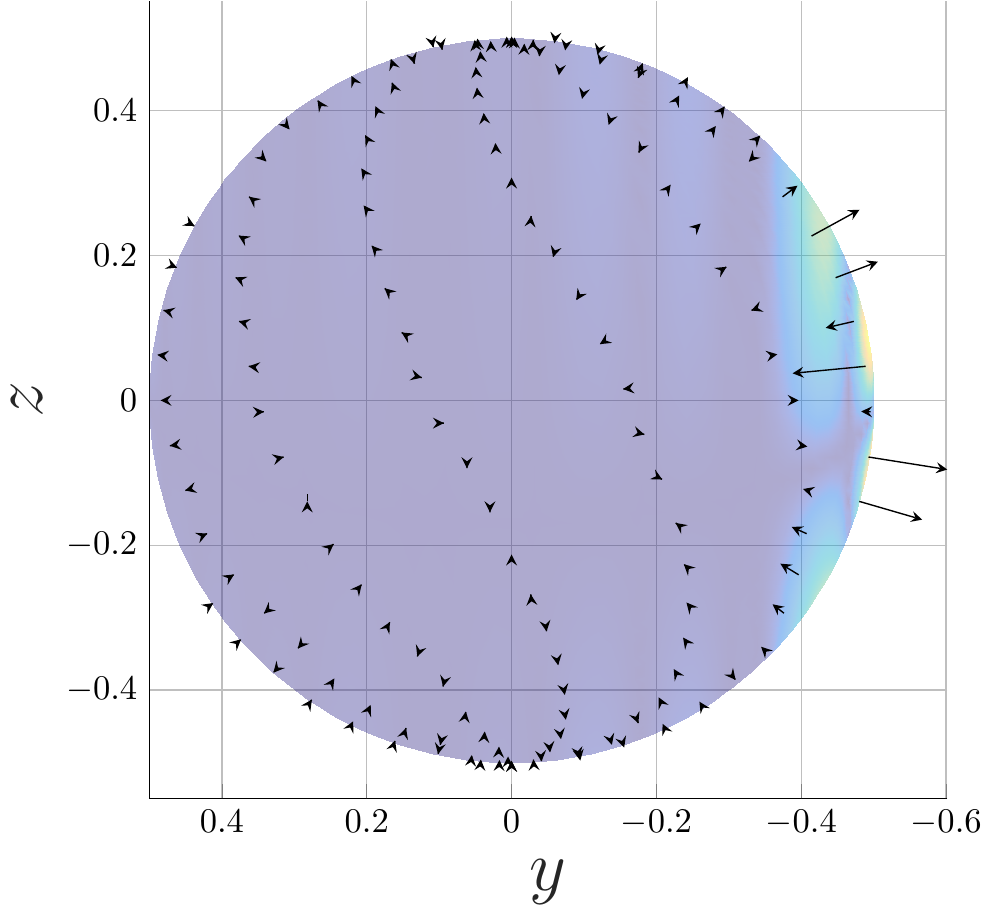}}
\caption{Vectorial Slepian functions on the ball from system I. Here, a normal field with related eigenvalue $0.909985$ is given. The representation of the Euclidean norm in the interior of the ball is shown. Blue depicts values close to zero and yellow stands for large values. Moreover, the vectorial functions are shown on a sphere with radius $0.5$ (right).}
\label{SysIRotRad}

\subfigure{\includegraphics[width=0.35\textwidth]{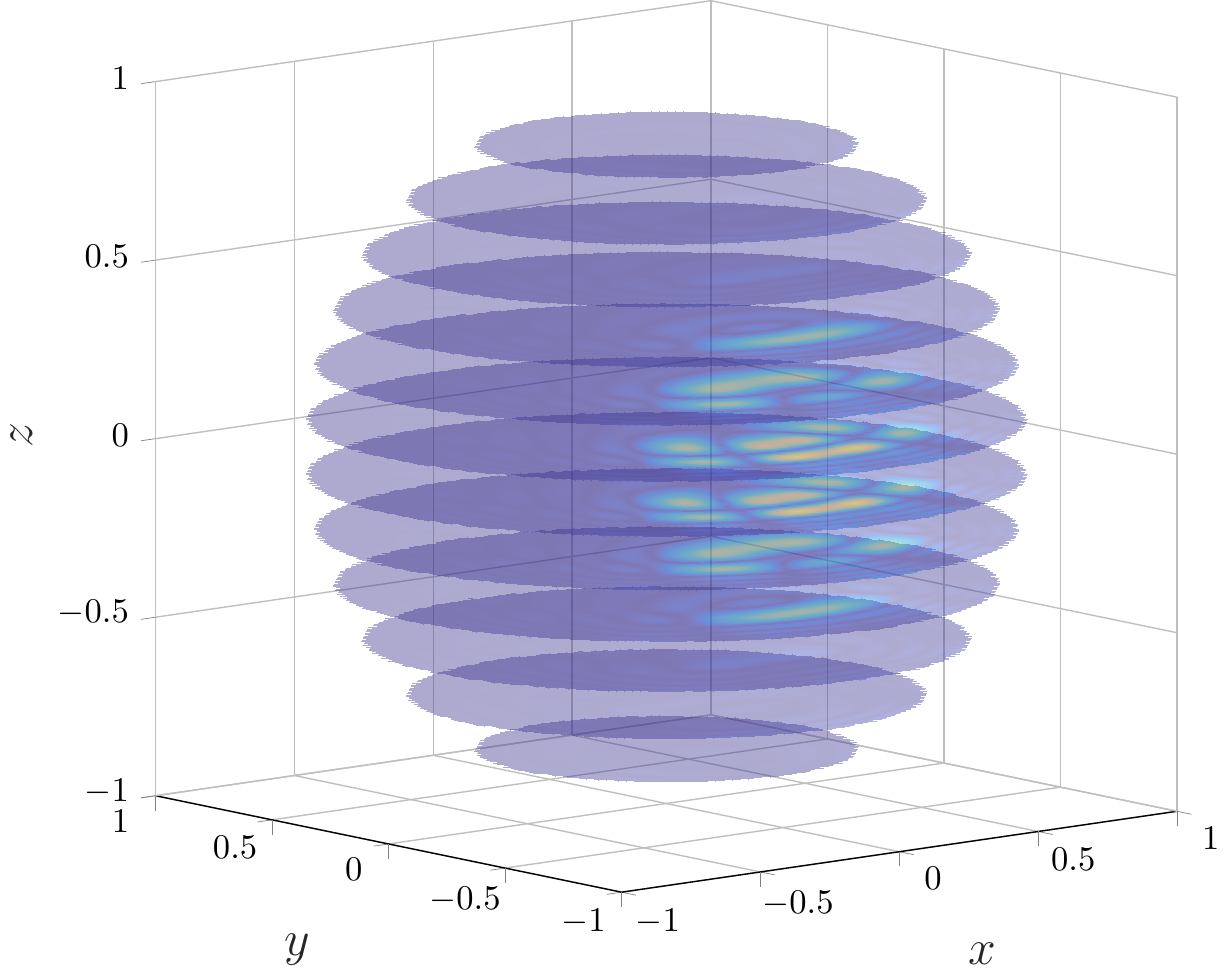}}
\subfigure{\includegraphics[width=0.27\textwidth]{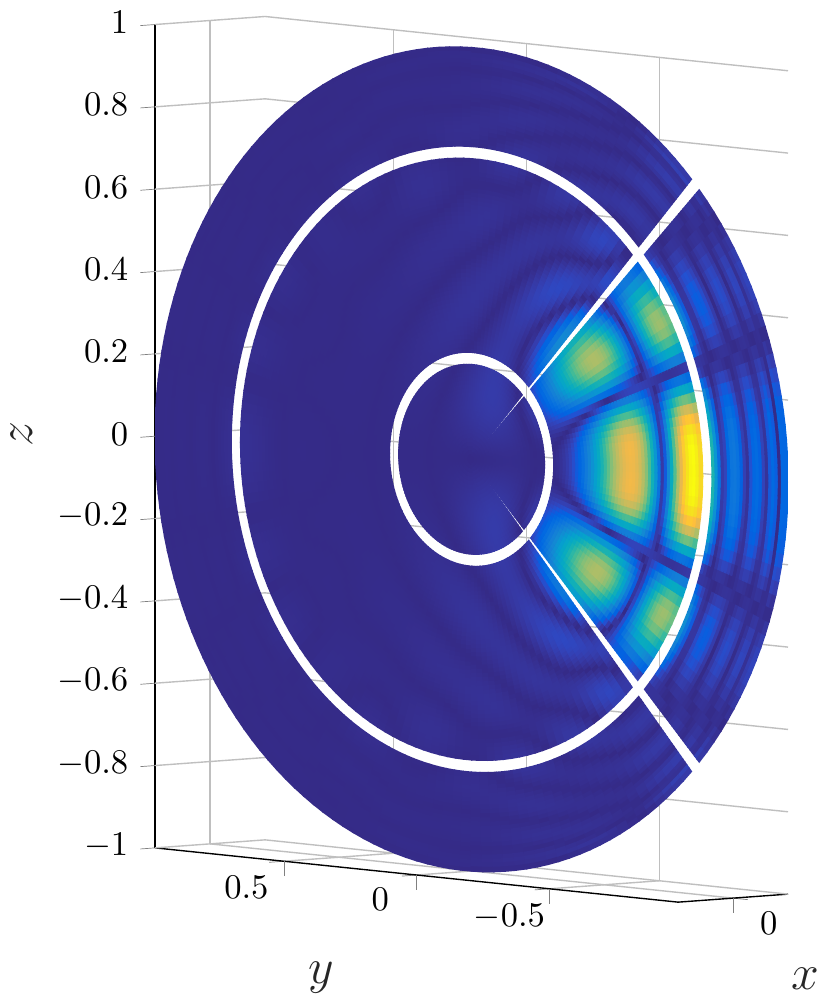}}
\subfigure{\includegraphics[width=0.35\textwidth]{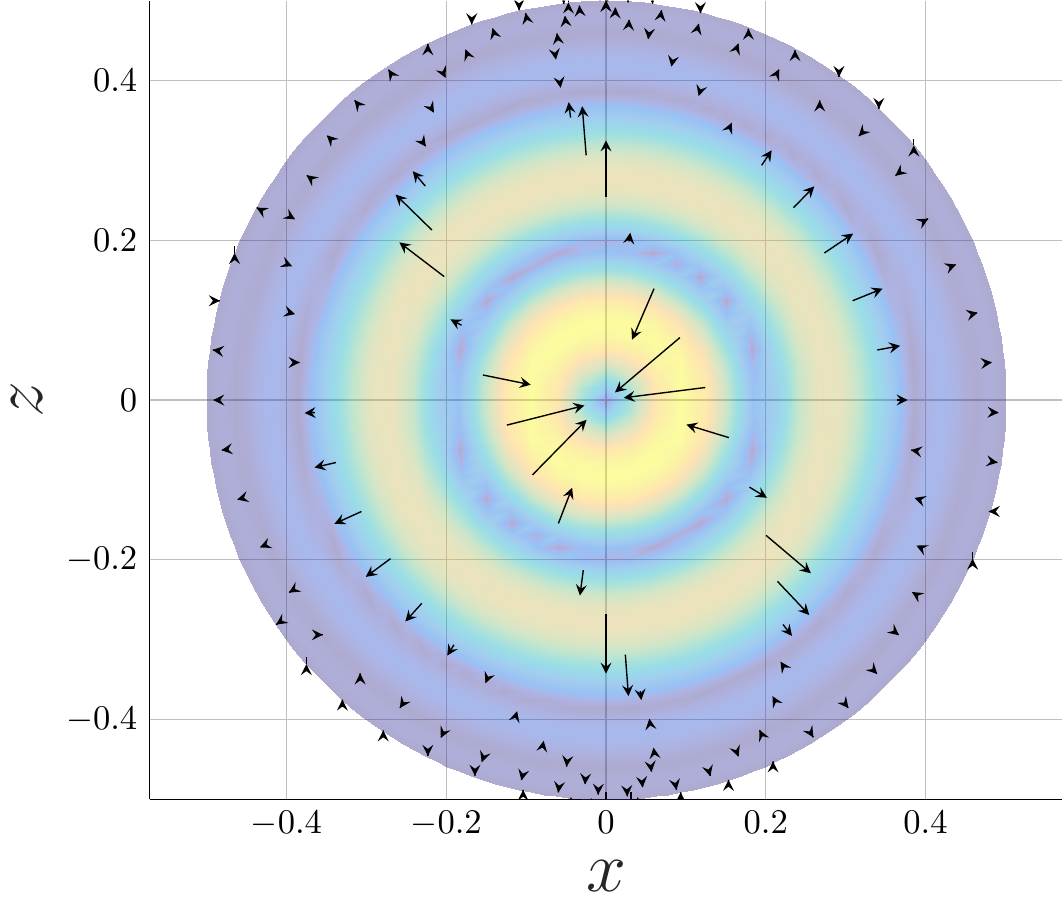}}
\caption{Vectorial Slepian functions on the ball from system I. Here, a tangential field with related eigenvalue $0.909980$ is given. The function is illustrated in a similar manner as in \cref{SysIRotRad}. Note the rotated coordinate system on the right-hand side (for a better visibility of the vectors).}
\label{SysIRotTan}
\end{figure}

Next, some numerical results are presented. For this purpose, some notes on the implementation and the setting of the experiments are made. After that the distribution of the eigenvalues is discussed and the constructed functions, the Shannon number as well as the rotated vector fields are evaluated.

First of all, note the order of the presented results. The results are sorted by the chosen systems. This means that \crefrange{EtaSysI}{SysITan} belong to system I and include the distribution of the eigenvalues, the Shannon number as well as a presentation of a normal field and a tangential field with large eigenvalues from different perspectives. In analogy, \crefrange{EtaSysII}{SysIITan} consider system II, \crefrange{EtaSysIII}{SysIIITan} belong to system III and \crefrange{SysIRotRad}{SysIRotTan} show some rotated results of system I. 

On subsets of $\ball$, the normal field of type 1 is orthonormal to both tangential fields of types 2 and 3 of any system. Thus, the implementation solves the problems independently. It follows the lines of the description of solving the localisation problem given above. In the following, we use the GNU Scientific Library (GSL) as documented in \cite{GSLdocu}. All remaining one-dimensional integrals are computed with the so-called `QAG adaptive integration' which is an adaptive integration method using Gauß-Kronrod quadrature formulae, see \cite[Sect.~17.3]{GSLdocu}. The term `QAG' defines that it is a quadrature routine (Q) with an adaptive integrator (A) and a user-defined general integrand (G), see \cite[p.~192]{GSLdocu}. We set the absolute error limit as well as the relative error limit to $10^{-12}$. We allow a maximum of $1000$ subintervals and use a 61 point Gauß-Kronrod rule in each subinterval. For further literature on adaptive integration methods or the Gauß-Kronrod rule see, for instance, \cite{adaptiveInt1,GaussKronrod1}, \cite[Sect.~42]{HankeBourg} and \cite[Sect.~7.5]{Schwarz}.

The Jacobi polynomials are highly oscillating in the origin in our setting. Throughout our research, it turned out that a non-adaptive integration method does not seem to be able to cope with these oscillations very well. An adaptive integration method, however, refines the integration grid autonomously in such areas. With the use of quadrature formulae with a high degree of exactness, this ansatz yields far more accurate integral values.

With these values, the localisation matrix is composed. Next, the eigenvalue problem is solved using the respective method implemented in the GSL, see \cite[Sect.~15.1]{GSLdocu}. For this type of matrices, a symmetric bidiagonalisation in combination with a QR reduction method can be applied, see, for instance, \cite[Sect.~5.5.3]{Schwarz} and \cite[Sect.~8.3]{GolubVanLoan}.  

After that, the Fourier coefficients of the additional rotated version of the Slepian functions are calculated. This is done using Wigner rotation matrices for vector fields. It follows the line of, for instance, \cite[App.~C.8]{Tromp1998}. At last, the values of the (rotated and non-rotated) vectorial Slepian functions on the ball at certain points in the ball are calculated with the use of their Fourier expansion. 

Recall that the bandlimits are denoted by $M \in \uint_0$ and $N \in \uint_{0_i}$ and the boundaries of the original partial cone are given by $a,\ b \in [0, \beta]$ for the radius $\beta$ of the ball and $\Theta \in [0, \pi]$. In the sequel, these parameters have to be fixed for the numerical experiments. Independent of the system, the choices for the parameters or `settings' \space of the experiments are given as
$$ M = 6, \quad N = 12, \quad a = 0.25, \quad b = 0.75, \quad \Theta = 45^\circ = \frac{\pi}{4}.$$ The ball is always chosen to be the unit ball $\ball$, i.e.\ the radius is $\beta =1$. Thus, \cref{OPC} pictures the setting of the localisation region. However, note that some pictures illustrate the functions on a sphere. For this, the sphere with radius $0.5$ is chosen for the plotting.
For experiments with respect to the Shannon number, the eigenvalues of the vectorial Slepian functions are calculated for diverse scenarios with the angle of the localisation region $\Theta$ varying between $15^\circ,\ 25^\circ,\ 35^\circ$ and $ 45^\circ$. However, the choices of $M,\ N,\ a$ and $b$ remain fixed. At last, some results of system I are also presented with respect to a rotation of the localisation region where all three Euler angles are chosen as $\frac{\pi}{2}.$  

\paragraph{The distribution of the eigenvalues}

The construction of the vectorial Slepian functions can be divided into the calculation of the normal fields and the tangential fields. The matrices $P^\star$ and $Q^\star$, as seen in \cref{LocMat}, stand for the respective problems. 
The distributions of the eigenvalues are depicted in \cref{EtaSysI,EtaSysII,EtaSysIII}. In each illustration, the values are sorted in descending order. 

The results of $P^\star$ as well as $Q^\star$ are generally as expected. The eigenvalues are situated in the interval $[0,1]$. Further, both distributions show a rapid decrease for each underlying basis system. This is explained by the fact that the localisation region is small compared to the whole of the unit ball $\ball$ as seen in \cref{OPC}. Note that the number of used basis functions with respect to the normal fields is 1183. The tangential fields use 2352 basis functions.

\paragraph{Evaluation of the functions}
At first, a few general properties are remarked. It is noticed that functions related to higher eigenvalues have less extrema as seen for instance in \cref{SysIRad,SysITan} in comparison to \cref{SysIRotRad,SysIRotTan} where we purposely selected a Slepian function with a lower eigenvalue for each case.
Further, Slepian functions with one angular and one radial extremum were obtained in each experiment. These look very similar with respect to their Euclidean norm as seen, for example, in \cref{SysIIRad,SysIITan}.

With respect to the results of systems I and III, the visually best-localised vectorial Slepian function on the ball is not necessarily the one related to the highest eigenvalue. However, its related eigenvalue is close to the maximal eigenvalue of the respective normal or tangential problem. The examples presented here in \cref{SysIRad,SysITan,SysIIIRad,SysIIITan} show mainly functions related to one of the largest eigenvalues. 

Regarding the directions of the vector fields in \cref{SysIRad,SysITan,SysIIRad,SysIITan,SysIIIRad,SysIIITan}, the functions depict the usual normal as well as tangential fields. Note that, with respect to the normal fields, the directions are either outer or inner normal vectors of a respective sphere. However, in general, the tangential fields show some properties that are in need of an explanation. A general representation of a tangential Slepian vector field is given by 
\begin{align*}
f^\star (x) = \sum_{i=2}^3\ \sum_{m,n,j}^{M,N}\ f^\star_{i,m,n,j}\ g_{m,n,j}^{(\star,i)}(x), \qquad x \in \ball. 
\end{align*}
In contrast to this, some functions possess the properties of the directions of the vector spherical harmonics as pictured, for instance, in \cite[p.~34]{FenglerDiss}. Some tangential Slepian vector fields have a vanishing surface curl. For others, the surface divergence is zero. This is explained as follows. The localisation submatrix $Q^\mathrm{\star}$ can be rearranged into a block-diagonal structure. This is done similarly as presented in \cite{Simons2014}. This rearrangement yields
\setlength{\arraycolsep}{6pt}
\begin{align*} 
Q^\star = 
\begin{pmatrix} 
Q_0^\star &0&&\cdots&& 0 \\
0 & Q_{-1}^\star &0&&& \\
&0& Q_1^\star &0&&\vdots\\
\vdots&&\ddots& \ddots &\ddots& \\
&&&0& Q_{-N}^\star  &0\\
0 &&\cdots&&0& Q_N^\star \\
\end{pmatrix}.
\end{align*}
\setlength{\arraycolsep}{2pt}

The submatrices $Q^\star_j$ have the form
\begin{align*} 
Q^\star_j = \begin{pmatrix}
B^\star_j & D^\star_j \\ \left( D^\star_j \right)^\mathrm{T} & C^\star_j 
\end{pmatrix}
= \begin{pmatrix}
K^\star_{(2,m,n,j),(2,m',n',j)} & K^\star_{(2,m,n,j),(3,m',n',-j)} \\ K^\star_{(3,m',n',-j),(2,m,n,j)} & K^\star_{(3,m,n,j),(3,m',n',j)} 
\end{pmatrix}
\end{align*}
for fixed $j=-N, \dots, N$ and running $m,\ m' =0, \dots,M$ and $n,\ n' = |j|+\delta_{j0}, \dots, N$.
Regarding a submatrix $D^\star_j$, the factor $j$ is contained in the entries $K^\star_{(2,m,n,j),(3,m',n',j')} $ and $K^\star_{(3,m',n',j),(2,m,n,j')} $, respectively as seen in \cref{MatrixEntries}. If $j=0$, these matrix coefficients vanish, i.e.\ $D_0^\star = 0$. Hence, the submatrix $Q^\star_0$ for this case is again a block-diagonal matrix:
\begin{align*} 
Q^\star_0 = \begin{pmatrix}
K^\star_{(2,m,n,0),(2,m',n',0)} & 0 \\ 0 & K^\star_{(3,m,n,0),(3,m',n',0)} 
\end{pmatrix}
= \begin{pmatrix}
B^\star_0 & 0\\ 0 & C^\star_0 
\end{pmatrix}.
\end{align*}
Thus, one block provides Fourier coefficients for the basis functions $g_{m,n,0}^{(\star,2)}$. The other one yields the coefficients for $g_{m,n,0}^{(\star,3)}$. Hence, for some tangential fields, their surface curl vanishes. For others, the surface divergence is zero.

Next, the results for the different systems are described and compared. Among the obtained functions, every combination of radial and angular extrema were observed in the underlying numerical experiments. This means, there exist functions with one radial extremum and a large amount of angular extrema as well as vice versa and every combination in between. The results with respect to system I and III are very similar to each other. This is justified by the similarity of the used basis functions. 

\enlargethispage{.5cm}
The basis functions of system I contain the damping factor $\left( \tfrac{r}{\beta} \right)^n$, where $r$ denotes the radial variable. Clearly, the influence of this factor diminishes near the surface of the ball. This means that large values can be expected rather near or at the surface. Nonetheless, the obtained Slepian functions like those in \cref{SysIRad,SysITan} are well-localised in the specified region.

With respect to system III, the basis functions have two important influences. On the one hand, the factor $\left( \tfrac{r}{\beta} \right)^{n-1}$ has a similar effect on them as the respective one on system I. On the other hand, the basis functions $g_{m,0,0}^{(\mathrm{III}, i)}$ are singular at the origin. These basis functions are contained in the Fourier expansion of every normal Slepian function. To present an example, where the localisation region provides a higher degree of difficulty in the calculations, we chose a region in the interior of the ball that does not contain the origin. As \cref{SysIIIRad,SysIIITan} show, also system III provides well-localised vector fields.

\enlargethispage{.5cm}

Recall that system II contains no damping factor. Hence, the behaviour of the Jacobi polynomials $P_m^{(0,2)}$ at $-1$ has a major impact on the basis functions. Note that, in the used parameterisation, the argument $t=-1$ of $P_m^{(0,2)} (t)$ corresponds to the centre of the ball. Thus, the basis functions of system II attain high values in a neighbourhood of this point and are discontinuous for $n>0$ at the origin. The size of this neighbourhood decreases for increasing radial degrees of the functions. The illustrations in \cref{SysIIRad,SysIITan} show that, if less radial extrema are attained, the influence of the origin is hardly recognizable and can, therefore, be neglected. All in all, the presented figures show well-localised functions also with respect to system II.

The results of the vectorial Slepian functions on the ball can be summed up as follows. In general, well-localised functions on the ball are obtained. Among these functions, some have one radial and one angular extremum. Others attain several radial and/or angular extrema. The properties noticed in the illustrations are explained by the theoretical approach taken and are analogous to those of known Slepian functions on other domains. Note that the vectorial Slepians functions on the ball inherit the properties of the chosen basis system.

\paragraph{Evaluation of the Shannon number}

\begin{wraptable}{l}{3.5cm}
	\begin{tabular}{|r||r|r|r|}
	\hline
	 $\Theta$ & $S^\mathrm{I}$ & $S^\mathrm{II}$ & $S^\mathrm{III}$  \\
	\hline
	\hline
	$15^\circ$ & 20 & 22 & 21 \\
	\hline
	$25^\circ$ & 54 & 62 & 57\\ 
	\hline
	$35^\circ$ & 104 & 119 & 109\\ 
	\hline 
	$45^\circ$ & 169 & 193 & 177\\ 
	\hline
	\end{tabular} 
\caption{Approximate Shannon number of the diverse experiments.}
\label{ShannonNumberExp} 
\end{wraptable}

The experiments are chosen as described above. The results are as expected. \cref{EtaSysIShannon,EtaSysIIShannon,EtaSysIIIShannon} show the distribution of the eigenvalues in the various settings. Note that only the first 250 eigenvalues are shown. This is done to improve the visualisation of the Shannon number. With respect to the systems I, II and III, the distribution of the whole set of eigenvalues from each experiment regarding the Shannon number shows the same behaviour as seen in \cref{EtaSysI,EtaSysII,EtaSysIII}. One figure stands for one experiment. The eigenvalues of $P^\star$ and $Q^\star$ are not separated here. 

The approximate Shannon numbers of the experiments are given in \cref{ShannonNumberExp}. 
Obviously, the amount of well-localised vectorial Slepian functions on the ball increases if the size of the original partial cone increases as well. This is not surprising, because larger regions require more basis functions to cover the variability of all functions on such a subdomain. 
Moreover, the Shannon number corresponding to system II is larger than those for systems I and III, which becomes significant for bigger cones.
Further, the Shannon number draws a line between significant and insignificant eigenvalues at around 0.4 in each experiment. Hence, the Shannon number predicts the number of well-localised Slepian functions pretty well.  

\paragraph{Evaluation of the rotated vector fields}
In the presented experiment, the vectorial functions are rotated by equal Euler angles $\frac{\pi}{2}$.
The rotation of the vector fields only alters the Fourier coefficients of the Slepian functions. Therefore, the eigenvalues remain the same as in \cref{EtaSysI}. This means it suffices to calculate the Slepian functions only for a partial cone centred around the North Pole, as we did it here, and to compute then the Slepian functions for an arbitrarily located partial cone. The obtained functions are merely rotated versions of the functions on the original partial cone, due to the symmetry of the sphere.

 \cref{SysIRotRad,SysIRotTan} show some rotated functions of system I. Here, two functions with slightly lower eigenvalues are presented to give an idea of what vectorial Slepian functions on the ball can also look like. The properties, as for instance the directions of the vector fields, and the idiosyncrasies, as for example the influence of the surface of the ball on system I, of the functions certainly remain the same if they are rotated.

\section{Conclusion}

From the scope of certain tomographic problems on the ball such as the inverse MEG- and EEG-problem, the need for well-localised vectorial functions on the ball occurs. The difficulty therein is caused by the uncertainty principle. According to this, a function cannot be simultaneously perfectly localised in space and in frequency. Therefore, in this paper, bandlimited functions were constructed which are optimally space-localised.

This problem was approached in the following way. Three different vector bases on the ball were used to build the bandlimited Fourier expansion of a Slepian vector field. They all consist of Jacobi polynomials and vector spherical harmonics. These bases are called systems I, II and III. The localisation region was defined as a part of a cone whose apex is situated in the origin. Following the original approach towards Slepian functions, the optimisation problem was altered into a finite-dimensional algebraic eigenvalue problem. The entries of the corresponding matrix were treated analytically as far as possible. The eigenvalue problem decouples into a normal and a tangential problem. For the remaining integrals and the eigenvalue problems, the GNU Scientific Library was applied. The number of well-localised vector fields was estimated by a Shannon number before the actual computation of the functions. This estimate mainly depends on the maximal radial and angular degree of the basis functions as well as the size of the localisation region.

The results of this approach towards vectorial Slepian functions on the ball can be summed up as follows. Regarding all three systems, well-localised functions on the ball were obtained. Among these functions, some have one radial and one angular extremum. Others attain several radial and/or angular extrema. Their visible properties can be explained by the theoretical approach taken. System I provides Slepian functions which tend to be concentrated near the surface of the ball. On the other hand, the origin has a major impact on the functions of system II. The resulting Slepian vector fields are concentrated there, if the localisation region is chosen close to the origin. At last, system III is influenced by the origin as well as by the surface of the ball. As this system is similarly constructed as system I, its results resemble the ones obtained from system I.

Thus, regardless of the chosen system, the vectorial Slepian functions on the ball are well-localised such that 
they appear to be suitable for applications where local vectorial phenomena (like currents) are analysed or modelled. Furthermore, the Shannon number predicts the number of well-localised vectorial Slepian functions on the ball pretty well.

In a forthcoming application, we will address the numerical regularisation of the inverse EEG-MEG-problem based on greedy algorithms developed in \cite{Fischer, DoreenDiss, Michel2015Handbook, MichelOrzlowski2017, MichelTelschow2016,  RogerDiss}. These algorithms require the choice of a dictionary with suitable trial functions. This dictionary is typically chosen to be overcomplete. For instance, Slepian functions and their rotated versions for other cones could be combined to better reveal the sources of an inverse problem. In this respect, we hope to improve the localisation of the detected neural currents by the inclusion of our Slepian functions in the dictionary.

\paragraph{Acknowledgement}
The authors gratefully acknowledge the financial support by the German Research Foundation (DFG), projects MI 655/10-1 and MI 655/7-2.

\appendix

\section{The Entries of the Localisation Matrix} \label{EntriesApp}

Let the localisation region $R$ be the original partial cone as given in \cref{localRegion}. Further, the basis functions from \cref{gmnj} are used. For the spherical harmonics, fully normalised spherical harmonics as defined in \cref{FNSH} are chosen. The Slepian functions are bandlimited. This means, the radial and angular degrees $m$ and $n$, respectively, are finite, i.e.\ $m=0, \dots, M$ and $n=0_i,\dots,N$ for some $M \in \uint_0$ and $N \in \uint_{0_i}$. The order $j$ is bounded by $-n \leq j \leq n$ for each $n$. The entries of the corresponding localisation matrix $K^\star$ are given in \cref{MatrixEntries}. In this appendix, these values are derived.

For $\lon \in [0, 2\pi[$ and $t \in [-1, 1]$, an established local orthonormal basis of $\real^3$ in terms of spherical coordinates is given by 
\begin{align*}
\era ( \lon, t ) = \left( \begin{matrix} \sqrt{1-t^2} \cos ( \lon ) \\ \sqrt{1-t^2} \sin ( \lon ) \\ t \end{matrix} \right),\
\ephi ( \lon, t ) = \left( \begin{matrix} - \sin ( \lon ) \\ \cos ( \lon ) \\ 0 \end{matrix} \right),\
\ete ( \lon, t ) = \left( \begin{matrix} -t \cos ( \lon ) \\ -t \sin ( \lon ) \\ \sqrt{1-t^2} \end{matrix} \right). 
\end{align*} 
For an illustration, see \cite[p.~86]{Michel2013}. With this basis, formulations of the surface gradient operator $\nabla^*$ and the surface curl $\mathrm{L}^*$ are given by
\begin{align}
\nabla^* &= \ephi \frac{1}{\sqrt{1-t^2}} \pdervlon + \ete \sqrt{1-t^2} \pdervt \notag
\intertext{and}
\mathrm{L}^* &= -\ephi \sqrt{1-t^2} \pdervt + \ete \frac{1}{\sqrt{1-t^2}} \pdervlon\ , \label{SurfCurl}
\end{align}
see, for example, \cite[p.~38]{FreedenSchreiner2009}. The vector $\era$ can also be used for the polar coordinate representation $x = r\era(\lon,t)$ of $x \in \real^3$ in the polar coordinates $(r, \lon,t)$, confer \cref{PolarParam}. The Jacobian of this parameterisation equals $r^2$. It is obtained by straight forward calculations. Thus, the integral of an arbitrary function $F \colon R \to \real$ over the original partial cone $R \subset \ball$ is given by
\begin{align*}
\int_{R}^{} F(x)\ \mathrm{d}x = \int_{a}^{b}\ \int_{0}^{2\pi}\ \int_{\cos \left( \Theta \right)}^{1}\ F(x(r,\lon,t))\ r^2\ \mathrm{d}t\  \mathrm{d}\lon\ \mathrm{d}r.
\end{align*}

Now consider the localisation matrix $K^\star$. In general, the matrix entries have the form 
\begin{align*}
K^\star_{(i,m,n,j), (i',m',n',j')} = \int_{R}^{} g_{m,n,j}^{(\star,i)} (x) \cdot g_{m',n',j'}^{(\star,i')}(x)\ \mathrm{d}x
\end{align*}
as seen in \cref{LocMat}. Type 1 is orthonormal to both types 2 and 3 also on all subsets of $\ball$. Thus, the localisation problem decouples into a normal part for equal types $i = i' = 1$ and a tangential part for types $i, i' \in \{ 2,3 \}$ as also seen in \cref{LocMat}.
For arbitrary $i$ and $i'$, the entries of $K^\star$ can be formulated as
\begin{align}
K^\star_{(i,m,n,j), (i',m',n',j')} &= \int_{R}^{}\ g_{m,n,j}^{(\star,i)}(x) \cdot g_{m',n',j'}^{(\star,i')}(x)\ \mathrm{d}x \notag \\
&=  T_{m,m',n,n'}^\star\ \int_{0}^{2\pi}\ \int_{\cos \left( \Theta \right)}^{1}\ y_{n,j}^{(i)}(\xi(\lon, t)) \cdot y_{n',j'}^{(i')}(\xi(\lon, t))\ \mathrm{d}t\ \mathrm{d}\lon \label{RadTimesAngInt}
\end{align}  where $T_{m,m',n,n'}^\star$ depends on the choice of the system. It is given by
\begin{align*}
T_{m,m',n,n'}^{\mathrm{I}} &=  \sqrt{\frac{4m+2n+3}{\beta^3}}\ \sqrt{\frac{4m'+2n'+3}{\beta^3}} \notag \\
& {} \times \int_{a}^{b}\ P_m^{(0,n+1/2)}\left( \frac{2r^2}{\beta^2} - 1 \right)  P_{m'}^{(0,n'+1/2)}\left( \frac{2r^2}{\beta^2} - 1 \right) \left( \frac{r}{\beta} \right)^{n+n'}  r^2\ \mathrm{d}r\
\intertext{for system I,}
T_{m,m',n,n'}^{\mathrm{II}} &= \sqrt{\frac{2m+3}{\beta^3}}\ \sqrt{\frac{2m'+3}{\beta^3}}\  \int_{a}^{b}\ P_m^{(0,2)} \left( \frac{2r}{\beta} - 1 \right) P_{m'}^{(0,2)} \left( \frac{2r}{\beta} - 1 \right)  r^2\ \mathrm{d}r\
\intertext{for any choice of $n$ and $ n'$ with respect to system II and}
T_{m,m',n,n'}^{\mathrm{III}} &=  \sqrt{\frac{4m+2n+1}{\beta^3}}\ \sqrt{\frac{4m'+2n'+1}{\beta^3}}\ \notag \\
& {} \times \int_{a}^{b}\ P_m^{(0,n-1/2)}\left( \frac{2r^2}{\beta^2} - 1 \right) P_{m'}^{(0,n'-1/2)}\left( \frac{2r^2}{\beta^2} - 1 \right) \left( \frac{r}{\beta} \right)^{n+n'-2}   r^2\ \mathrm{d}r\
\end{align*} 
for system III. The substitution 
\begin{align} \label{SubsI}
r = \phi(u) = \beta\ \sqrt{\frac{u+1}{2}}\ , \quad \phi'(u) = \frac{\beta}{4}\ \sqrt{\frac{2}{u+1}}
\end{align} 
with respect to systems I and III, as well as the substitution
\begin{align} \label{SubsII}
r = \phi(u) = \frac{\beta(u+1)}{2}\ , \quad \phi'(u) = \frac{\beta}{2}
\end{align} 
with respect to system II, provides the formulation of $T_{m,m',n,n'}^\star$ as  

\begin{align*}
T_{m,m',n,n'}^{\mathrm{I}} &= \sqrt{\frac{(4m+2n+3)\ (4m'+2n'+3)}{2^{n+n'+5}}}\ \notag \\
& {} \times  \int_{\frac{2a^2}{\beta^2} -1}^{\frac{2b^2}{\beta^2} -1}\ P_m^{(0,n+1/2)}\left( u \right)  P_{m'}^{(0,n'+1/2)}\left( u \right) \left( u+1 \right)^{(n+n'+1)/2}\ \mathrm{d}u
\intertext{for system I,}
T_{m,m',n,n'}^{\mathrm{II}} &= \sqrt{ \frac{ (2m+3)\ (2m'+3)} {64} }\  \int_{\frac{2a}{\beta} -1}^{\frac{2b}{\beta} -1} P_m^{(0,2)} \left( u \right) P_{m'}^{(0,2)} \left( u \right) (u+1)^2\ \mathrm{d}u
\intertext{for any choice of $n$ and $n'$ with respect to system II and}
T_{m,m',n,n'}^{\mathrm{III}} &= \sqrt{\frac{ (4m+2n+1)\ (4m'+2n'+1) }{2^{n+n'+3}}}\ \notag \\
& {} \times \int_{\frac{2a^2}{\beta^2} -1}^{\frac{2b^2}{\beta^2} -1} P_m^{(0,n-1/2)}\left( u \right) P_{m'}^{(0,n'-1/2)}\left( u \right) \left( u+1 \right)^{(n+n'-1)/2}\ \mathrm{d}u
\end{align*} for system III. 
Note that with this formulation, $T_{m,m',n,n'}^\star$ obviously coincides with the term $a_{m,m',n,n'}^\star I_{m,m',n,n'}^\star$ from \cref{MatrixFactors}. Hence, we have \cref{RadTimesAngInt} in the form
\begin{align*}
K^\star_{(i,m,n,j), (i',m',n',j')} = a_{m,m',n,n'}^\star I_{m,m',n,n'}^\star\ \int_{0}^{2\pi}\ \int_{\cos \left( \Theta \right)}^{1}\ y_{n,j}^{(i)}(\xi(\lon, t)) \cdot y_{n',j'}^{(i')}(\xi(\lon, t))\ \mathrm{d}t\ \mathrm{d}\lon.
\end{align*}

Now consider the angular integral
\begin{align*} 
\int_{0}^{2\pi}\ \int_{\cos \left( \Theta \right)}^{1}\ y_{n,j}^{(i)}(\xi(\lon, t)) \cdot y_{n',j'}^{(i')}(\xi(\lon, t))\ \mathrm{d}t\ \mathrm{d}\lon.
\end{align*} 
Depending on the choice of $i$ and $i'$, there are four different cases to be discussed. In most of these cases, we need to consider the integral 
\begin{align}
\int_{0}^{2\pi}\  c_j(\lon) c_{j'}(\lon) \mathrm{d}\lon \coloneqq
\int_{0}^{2\pi}\ \left\{ \begin{matrix}
\sqrt{2}\cos \left(j\lon \right), & j < 0 \\
1, & j = 0\\
\sqrt{2}\sin \left(j\lon \right), & j > 0
\end{matrix} \right\}   
\left\{ \begin{matrix}
\sqrt{2}\cos \left(j'\lon \right), & j' < 0 \\
1, & j' = 0\\
\sqrt{2}\sin \left(j'\lon \right), & j' > 0
\end{matrix} \right\}\ \mathrm{d}\lon
= 2 \pi \delta_{jj'}
\label{trigInt}
\end{align}
with the same abbreviation as in \cref{FNSH} and the Kronecker Delta $\delta_{jj'}$.

\paragraph{Case 1}
At first, the case $i=i'=1$ is examined. Here, the integral is
\begin{align} 
\int_{0}^{2\pi}\ \int_{\cos \left( \Theta \right)}^{1}\ y_{n,j}^{(1)}(\xi(\lon, t)) \cdot y_{n',j'}^{(1)}(\xi(\lon, t))\ \mathrm{d}t\ \mathrm{d}\lon. \label{tmp18}
\end{align}
Inserting the definition of vector spherical harmonics of type 1 and of real fully normalised spherical harmonics as well as \cref{trigInt} in \cref{tmp18}, we obtain
\begin{align*}
\MoveEqLeft[3] \int_{0}^{2\pi}\ \int_{\cos \left( \Theta \right)}^{1}\ y_{n,j}^{(1)}(\xi(\lon, t)) \cdot y_{n',j'}^{(1)}(\xi(\lon, t))\ \mathrm{d}t\ \mathrm{d}\lon\\
= {} & \int_{0}^{2\pi}\ \int_{\cos \left( \Theta \right)}^{1}\ \left( \era Y_{n,j}(\xi(\lon, t)) \right) \cdot \left( \era Y_{n',j'}(\xi(\lon, t)) \right)\ \mathrm{d}t\ \mathrm{d}\lon \\
= {} & \int_{0}^{2\pi}\ \int_{\cos \left( \Theta \right)}^{1}\ Y_{n,j}(\xi(\lon, t))\ Y_{n',j'}(\xi(\lon, t))\ \mathrm{d}t\ \mathrm{d}\lon \\
= {} & \int_{0}^{2\pi}\ \int_{\cos \left( \Theta \right)}^{1}\ \left( b_{n,j} P_{n,|j|}(t)\ \frac{1}{\sqrt{2\pi}}\ c_j(\lon) \right) \left( b_{n',j'} P_{n',|j'|}(t)\ \frac{1}{\sqrt{2\pi}}\  c_{j'}(\lon) \right)\ \mathrm{d}t\ \mathrm{d}\lon  \\
= {} & b_{n,j} b_{n',j'} \int_{\cos \left( \Theta \right)}^{1}\ P_{n,|j|}(t)\ P_{n',|j'|}(t)\ \mathrm{d}t \ \frac{1}{2\pi}\ \int_{0}^{2\pi}\ c_{j}(\lon) c_{j'}(\lon)\ \mathrm{d}\lon \\
= {} & \delta_{jj'} b_{n,j} b_{n',j} \int_{\cos \left( \Theta \right)}^{1}\ P_{n,|j|}(t)\ P_{n',|j|}(t)\ \mathrm{d}t.
\end{align*}
In combination with the radial part, this yields the representation of the matrix entries $K^\star_{(1,m,n,j),(1,m',n',j')}$ as given in \cref{MatrixEntries}:
\begin{align*} 
\MoveEqLeft[3] K^{\star}_{(1,m,n,j),(1,m',n',j')} = a_{m,m',n,n'}^\star I_{m,m',n,n'}^\star  b_{n,j} b_{n',j} \delta_{jj'} \int_{\cos \left( \Theta \right)}^{1}\ P_{n,|j|}(t)\ P_{n',|j|}(t)\ \mathrm{d}t.
\end{align*}

\paragraph{Case 2}
The second case deals with $i$ and $i'$ both of type 2, this means the integral of the form
\begin{align} 
\int_{0}^{2\pi}\ \int_{\cos \left( \Theta \right)}^{1}\ y_{n,j}^{(2)}(\xi(\lon, t)) \cdot y_{n',j'}^{(2)}(\xi(\lon, t))\ \mathrm{d}t\ \mathrm{d}\lon. \label{tmp19}
\end{align}
The first simplifications are made with the use of the definition of vector spherical harmonics, the definition of their normalisation factor and Green's first surface identity:
\begin{align}
\MoveEqLeft[3] \int_{0}^{2\pi}\ \int_{\cos \left( \Theta \right)}^{1}\ y_{n,j}^{(2)}(\xi(\lon, t)) \cdot y_{n',j'}^{(2)}(\xi(\lon, t))\ \mathrm{d}t\ \mathrm{d}\lon \notag\\
= {} & \frac{1}{\sqrt{n (n+1) n' (n'+1)}} \int_{0}^{2\pi}\ \int_{\cos \left( \Theta \right)}^{1}\ \left( \nabla^\ast_\xi\ Y_{n,j}(\xi(\lon, t)) \right) \cdot \left( \nabla^\ast_\xi\ Y_{n',j'}(\xi(\lon, t)) \right)\ \mathrm{d}t\ \mathrm{d}\lon \notag\\
= {} & \frac{1}{\sqrt{n (n+1) n' (n'+1)}} \left(- \int_{0}^{2\pi}\ \int_{\cos \left( \Theta \right)}^{1}\ Y_{n,j}(\xi(\lon, t))\ \Delta_\xi^\ast\ Y_{n',j'}(\xi(\lon, t))\ \mathrm{d}t\ \mathrm{d}\lon \right. \notag\\ 
& {} \qquad + \left. \displaystyle\int_{\partial\mathcal{C}}^{}\ Y_{n,j}(\xi)\ \pdervnu\ Y_{n',j'}(\xi)\ \mathrm{ds}(\xi) \vphantom{\int_{0}^{2\pi}\ \int_{\cos \left( \Theta \right)}^{1}} \right),
\label{tmp50}
\end{align}
where $\partial\mathcal{C}$ indicates the boundary of the spherical cap and $\nu$ is the corresponding outer unit normal vector. Note that the spherical harmonics of degree $n'$ are the eigenfunctions of the Beltrami operator to the eigenvalue $-n'(n'+1)$, see, for example, \cite[pp.~123-124]{Michel2013}. Hence, the first summand rearranges to 
\begin{align}
\sqrt{\frac{n' (n'+1)}{n (n+1)}}\ \int_{0}^{2\pi}\ \int_{\cos \left( \Theta \right)}^{1}\  Y_{n,j}(\xi(\lon, t))\ Y_{n',j'}(\xi(\lon, t))\ \mathrm{d}t\ \mathrm{d}\lon.
\end{align}
The boundary $\partial\mathcal{C}$ is parameterised by
\begin{align} 
g \colon [0,2\pi] \to \real^3, \qquad \lon \mapsto \Big( \sin \left( \Theta \right) \cos (\lon),\ \sin \left( \Theta \right) \sin (\lon),\ \cos ( \Theta ) \Big)^\mathrm{T}.\label{tmp27}
\end{align}
The Euclidean norm of its derivative equals $\sin \left( \Theta \right)$. Further, the derivative of $Y_{n',j'}$ along the outer normal $\nu$ of $\partial\mathcal{C}$ is given by
\begin{align*}
\MoveEqLeft[3] \pdervnu\ Y_{n',j'}(\xi(\lon, t)) \\
= {} & \nu \cdot \nabla^\ast_\xi\ Y_{n',j'}(\xi(\lon, t)) \\
= {} & -\ete \cdot \nabla^\ast_\xi\ Y_{n',j'}(\xi(\lon, t))\\
= {} & -\ete \cdot \left( \ephi\ \frac{1}{\sqrt{1-t^2}}\ \pdervlon\ Y_{n',j'}(\xi(\lon, t)) + \ete\ \sqrt{1-t^2}\ \pdervt\ Y_{n',j'}(\xi(\lon, t))\ \right) \\
= {} & - \sqrt{1-t^2}\ \pdervt\ Y_{n',j'}(\xi(\lon, t)).
\end{align*}
The first step is based on the equality $\pdervnu = \nu \cdot \nabla^\ast_\xi$ as given in \cite[p.~41]{FreedenSchreiner2009}. In the case of the original partial cone, the outer normal of $\partial\mathcal{C}$ equals $-\ete$. At the boundary of the spherical cap, the polar distance $t$ attains the value $\cos \left( \Theta \right)$. Hence, the normal derivative is given by 
\begin{align*}
\left. \left( \pdervnu\ Y_{n',j'}(\xi(\lon,t)) \right) \right|_{t=\cos \left( \Theta \right)} = - \sin \left( \Theta \right) \left. \left( \pdervt\ Y_{n',j'}(\xi(\lon,t)) \right) \right|_{t=\cos \left( \Theta \right)}.
\end{align*}
Therefore, \cref{tmp50} is rearranged to
\begin{align}
& \qquad \sqrt{\frac{n' (n'+1)}{n (n+1)}}\ \int_{0}^{2\pi}\ \int_{\cos \left( \Theta \right)}^{1}\  Y_{n,j}(\xi(\lon, t))\ Y_{n',j'}(\xi(\lon, t))\ \mathrm{d}t\ \mathrm{d}\lon \notag\\ 
& - \frac{\sin^2 \left( \Theta \right)}{\sqrt{n (n+1) n' (n'+1)}}\ \int_{0}^{2\pi}\ Y_{n,j}(\xi(\lon, \cos \left( \Theta \right)))\ \left. \left( \pdervt\ Y_{n',j'}(\xi(\lon, t))\right)\right|_{t=\cos \left( \Theta \right)}\ \mathrm{d}\lon. \label{tmp20}
\end{align}
Thus, a preliminary form of \cref{tmp19} is provided by \cref{tmp20}. The integral of the first summand in \cref{tmp20} is examined in case 1. It has the form 
\begin{align*}
\MoveEqLeft[3] \sqrt{\frac{n' (n'+1)}{n (n+1)}}\ \int_{0}^{2\pi}\ \int_{\cos \left( \Theta \right)}^{1}\  Y_{n,j}(\lon,t)\ Y_{n',j'}(\lon,t)\ \mathrm{d}t\ \mathrm{d}\lon \\
& {} = \delta_{jj'} b_{n,j} b_{n',j}\ \sqrt{\frac{n' (n'+1)}{n (n+1)}}  \int_{\cos \left( \Theta \right)}^{1}\ P_{n,|j|}(t)\ P_{n',|j|}(t)\ \mathrm{d}t.
\end{align*}
The integral of the second summand needs to be discussed further. If the fully normalised spherical harmonics are inserted, the integral is given by
\begin{align*}
\MoveEqLeft[3] \int_{0}^{2\pi}\ Y_{n,j}(\xi(\lon, \cos \left( \Theta \right))) \left. \left( \pdervt\ Y_{n',j'}(\xi(\lon,t)) \right) \right|_{t=\cos \left( \Theta \right)}\ \mathrm{d}\lon\ \notag \\
= {} & \int_{0}^{2\pi}\ \left( b_{n,j} P_{n,|j|}(\cos \left( \Theta \right))  \frac{1}{\sqrt{2\pi}} c_j(\lon) \right)  \pdervt\ \left. \left( b_{n',j'} P_{n',|j'|}(t) \frac{1}{\sqrt{2\pi}} c_{j'}(\lon) \right) \right|_{t=\cos \left( \Theta \right)}\ \mathrm{d}\lon \\ 
= {} & b_{n,j} b_{n',j'} P_{n,|j|}(\cos \left( \Theta \right))  P'_{n',|j'|}(\cos \left( \Theta \right))\ \frac{1}{2\pi} \int_{0}^{2\pi}\ c_j(\lon) c_{j'}(\lon)\ \mathrm{d}\lon.
\end{align*}
Again, the use of \cref{trigInt} yields a Kronecker delta of $j$ and $j'$ for the latter integral. Hence, the second summand in \cref{tmp20} is given by
\begin{align*}
\MoveEqLeft[3] - \frac{\sin^2 \left( \Theta \right)}{\sqrt{n (n+1) n' (n'+1)}}\ \int_{0}^{2\pi}\ Y_{n,j}(\xi(\lon, \cos \left( \Theta \right)))\ \left. \left( \pdervt\ Y_{n',j'}(\xi(\lon, t)) \right) \right|_{t=\cos \left( \Theta \right)}\ \mathrm{d}\lon \notag \\
= {} &  - \delta_{jj'}\ \frac{\sin^2 \left( \Theta \right)}{\sqrt{n (n+1) n' (n'+1)}}\ b_{n,j} b_{n'j} P_{n,|j|}(\cos \left( \Theta \right))  P'_{n',|j|}(\cos \left( \Theta \right)). 
\end{align*}
All in all, this provides the form of the entries $K^\star_{(2,m,n,j),(2,m',n',j')}$ as given in \cref{MatrixEntries}:
\begin{align*} 
\MoveEqLeft[3] K^{\star}_{(2,m,n,j),(2,m',n',j')} = a_{m,m',n,n'}^\star I_{m,m',n,n'}^\star  b_{n,j} b_{n',j}  \delta_{jj'}\\
& {} \times \left( \sqrt{\frac{n' (n'+1)}{n (n+1)}}\ \int_{\cos \left( \Theta \right)}^{1}\ P_{n,|j|}(t)\ P_{n',|j|}(t)\ \mathrm{d}t \right. \\
& {} \quad - \left. \vphantom{\int_{\cos \left( \Theta \right)}^{1}} \frac{\sin^2 \left( \Theta \right)}{\sqrt{n (n+1) n' (n'+1)}}\ P_{n,|j|}(\cos \left( \Theta \right))\ P'_{n',|j|}(\cos \left( \Theta \right)) \right).
\end{align*}

\paragraph{Case 3}
The third case is the last case with equal $i$ and $i'$, i.e.\ $i=i'=3$. It discusses integrals of the form 
\begin{align*}
\int_{0}^{2\pi}\ \int_{\cos \left( \Theta \right)}^{1}\ y_{n,j}^{(3)}(\xi(\lon,t)) \cdot y_{n',j'}^{(3)}(\xi(\lon,t))\ \mathrm{d}t\ \mathrm{d}\lon.
\end{align*}

The equality of Case 2 and Case 3 was already seen in \cite[eq.~(46)]{Simons2014}. Hence, this setting yields equal submatrices $B^\star$ and $C^\star$ from \cref{LocMat} of $K^\star$:
\begin{align*}
K_{(3,m,n,j),(3,m',n',j')}  = K_{(2,m,n,j),(2,m',n',j')}.
\end{align*}


\paragraph{Case 4}
At last, the case of distinct types $i$ and $i'$ is considered. Due to the orthonormality of the basis functions from \cref{gmnj}, the mixed case deals with $i=2$ and $i'=3$. Note that the opposite case $i=3$ and $i'=2$ produces the transposed submatrix of this case because of the symmetry of the Euclidean inner product in \cref{LocMat}. This means, integrals of the form 
\begin{align*}
\int_{0}^{2\pi}\ \int_{\cos \left( \Theta \right)}^{1}\ y_{n,j}^{(2)}(\xi(\lon, t)) \cdot y_{n',j'}^{(3)}(\xi(\lon, t))\ \mathrm{d}t\ \mathrm{d}\lon
\end{align*}
are considered at this point. Inserting the definition of the vector spherical harmonics, we have
\begin{align*}
\frac{1}{\sqrt{n (n+1) n' (n'+1)}} \int_{0}^{2\pi}\ \int_{\cos \left( \Theta \right)}^{1}\ \nabla^*_\xi Y_{n,j}(\xi(\lon, t)) \cdot \mathrm{L}^*_\xi Y_{n',j'} (\xi(\lon, t))\ \mathrm{d}t\ \mathrm{d}\lon.
\end{align*}
For the next considerations, we abbreviate the quotient upfront with 
\begin{align}
c_{nn'} \coloneqq \frac{1}{\sqrt{n (n+1) n' (n'+1)}}\ .
\end{align}
We can extend the integral due to the orthogonality of the surface gradient operator and the surface curl, see for example \cite[p.~39,~(2.142)]{FreedenSchreiner2009}, and obtain
\begin{align*}
c_{nn'} & \left( \int_{0}^{2\pi}\ \int_{\cos \left( \Theta \right)}^{1}\ \nabla^*_\xi Y_{n,j}(\xi(\lon, t)) \cdot \mathrm{L}^*_\xi Y_{n',j'}(\xi(\lon, t))\ \mathrm{d}t\ \mathrm{d}\lon \right. \\ 
& \quad + \left. \int_{0}^{2\pi}\ \int_{\cos \left( \Theta \right)}^{1}\ Y_{n,j}(\xi(\lon, t))\ \nabla^*_\xi \cdot \mathrm{L}^*_\xi Y_{n',j'}(\xi(\lon, t))\ \mathrm{d}t\ \mathrm{d}\lon \right).
\end{align*}
With $F=Y_{n,j}$ and $g = \mathrm{L}^*_\xi Y_{n',j'}$, this equals
\begin{align*}
c_{nn'}  & \left( \int_{\mathcal{C}} ^{} \nabla^*_\xi F(\xi) \cdot g(\xi)\ \mathrm{d}\omega(\xi) + \int_{\mathcal{C}} F(\xi)\ \nabla^*_\xi \cdot g(\xi)\ \mathrm{d} \omega(\xi) \right) \\
&= c_{nn'}  \int_{\mathcal{C}} ^{} \nabla^*_\xi \cdot \Big( F(\xi)\  g(\xi) \Big)\ \mathrm{d}\omega(\xi)
\end{align*}
where $\mathcal{C}$ is the considered spherical cap. Due to the surface theorem of Gauß, see for example \cite[p.~116]{FreedenGutting2013}, we have
\begin{align*}
c_{nn'}  &\int_{\mathcal{C}} ^{} \nabla^*_\xi \cdot \Big( F(\xi) g(\xi) \Big)\ \mathrm{d}\omega(\xi) 
=  c_{nn'} \int_{\partial \mathcal{C}} ^{} \nu \cdot \Big( F(\xi) g(\xi) \Big)\ \mathrm{ds}(\xi)\\
&= c_{nn'}  \int_{\partial \mathcal{C}} ^{} F(\xi)\ g(\xi) \cdot \nu \ \mathrm{d}s(\xi).
\end{align*}
Hence, in our case, we have
\begin{align*}
\int_{0}^{2\pi}\ &\int_{\cos \left( \Theta \right)}^{1}\ y_{n,j}^{(2)}(\xi(\lon, t)) \cdot y_{n',j'}^{(3)}(\xi(\lon, t))\ \mathrm{d}t\ \mathrm{d}\lon \\
&= c_{nn'}  \int_{\partial \mathcal{C}}^{} Y_{n,j}(\xi)\ \left( \mathrm{L}^*_\xi Y_{n',j'}(\xi) \right) \cdot  \nu \ \mathrm{d} s(\xi).
\end{align*}
As mentioned before, the outer normal $\nu$ of the boundary of the spherical cap equals $-\ete$. Due to the local coordinate representation of the surface curl \cref{SurfCurl} and with the parameterisation of $\partial \mathcal{C}$ from \cref{tmp27}, we obtain
\begin{align*}
c_{nn'}  &\int_{\partial \mathcal{C}}^{} Y_{n,j}(\xi)\ \left( \mathrm{L}^*_\xi Y_{n',j'}(\xi) \right) \cdot  \nu \ \mathrm{d} s(\xi) \\
= {} & - c_{nn'}  \int_{\partial \mathcal{C}}^{} Y_{n,j}(\xi)\ \frac{1}{\sqrt{1-t^2}}\ \pdervlon\ Y_{n',j'}(\xi(\lon, t))\ \mathrm{d} s(\xi) \\
= {} & - c_{nn'} \int_{0}^{2\pi} Y_{n,j}(\xi(\lon, \cos \left( \Theta \right) ))\ \pdervlon\ Y_{n',j'}(\xi(\lon, \cos \left( \Theta \right)))\ \mathrm{d} \lon.
\end{align*}
The definition of the fully normalised spherical harmonics \cref{FNSH} shows that $$\pdervlon\ Y_{n',j'}(\xi(\lon, \cos \left( \Theta \right))) = j'\ Y_{n',-j'}(\xi(\lon, \cos \left( \Theta \right)).$$ Thus, we have 
\begin{align*}
\MoveEqLeft[3] - c_{nn'}  \int_{0}^{2\pi} Y_{n,j}(\xi(\lon, \cos \left( \Theta \right) ))\ \pdervlon\ Y_{n',j'}(\xi(\lon, \cos \left( \Theta \right)))\ \mathrm{d} \lon\\
= {} & - c_{nn'}  \int_{0}^{2\pi} Y_{n,j}(\xi(\lon, \cos \left( \Theta \right) ))\ j'\ Y_{n',-j'}(\xi(\lon, \cos \left( \Theta \right)))\ \mathrm{d} \lon.
\end{align*}
At this point, we can insert the definition of the fully normalised spherical harmonics.
\begin{align*}
\MoveEqLeft[3] - c_{nn'}  \int_{0}^{2\pi} Y_{n,j}(\xi(\lon, \cos \left( \Theta \right) ))\ j' Y_{n',-j'}(\xi(\lon, \cos \left( \Theta \right)))\ \mathrm{d} \lon \\
= {} & - \frac{j'}{\sqrt{n (n+1) n' (n'+1)}} \int_{0}^{2\pi} b_{n,j}\ P_{n,|j|} ( \cos \left(\Theta \right))\ \frac{1}{\sqrt{2\pi}}\ c_j(\lon) \\
& \hspace*{4.5cm} \times \  b_{n',-j'}\ P_{n,|-j'|} ( \cos \left(\Theta \right))\ \frac{1}{\sqrt{2\pi}}\ c_{-j'}(\lon)\  \mathrm{d} \lon \\
= {} & - \frac{j'}{\sqrt{n (n+1) n' (n'+1)}}\ b_{n,j}\ b_{n',-j'}\ P_{n,|j|} ( \cos \left(\Theta \right))\  P_{n,|-j'|} ( \cos \left(\Theta \right)) \\
& \hspace*{4.5cm} \times \  \frac{1}{2\pi}\ \int_{0}^{2\pi} c_j(\lon)\ c_{-j'}(\lon)\  \mathrm{d} \lon .
\end{align*}
With the use of \cref{trigInt}, the last line equals
\begin{align*}
\MoveEqLeft[3] - \frac{j' \delta_{j,-j'}}{\sqrt{n (n+1) n' (n'+1)}}\ b_{n,j}\ b_{n',-j'}\ P_{n,|j|} ( \cos \left(\Theta \right))\  P_{n,|-j'|} ( \cos \left(\Theta \right)) \\
= {} & \frac{j \delta_{-j,j'}}{\sqrt{n (n+1) n' (n'+1)}}\ b_{n,j}\ b_{n',-j'}\ P_{n,|j|} ( \cos \left(\Theta \right))\  P_{n,|-j'|} ( \cos \left(\Theta \right)).
\end{align*}
In combination with the respective radial integral, this yields the matrix entries as given in \cref{MatrixEntries} (note that $b_{n',-j'} = b_{n',j'}$):
\begin{align*} 
\MoveEqLeft[3] K^{\star}_{(2,m,n,j),(3,m',n',j')} \notag \\ 
= {} & a_{m,m',n,n'}^\star I_{m,m',n,n'}^\star b_{n,j}\ b_{n',-j}\  \frac{j\delta_{-j,j'}}{\sqrt{n (n+1) n' (n'+1)}}\ P_{n,|j|}(\cos \left( \Theta \right))\  P_{n',|j'|}(\cos \left( \Theta \right)).
\end{align*}

Hence, \cref{MatrixEntries} is proven. $\hfill \Box$

\bibliographystyle{abbrv}
\bibliography{literature}

\end{document}